\documentclass{article}
\usepackage{amsmath, amsfonts, amsthm, amssymb}
\usepackage{graphicx}
\usepackage{float}
\usepackage{verbatim}
\usepackage{color}
\usepackage[sort]{cite}

\numberwithin{equation}{section}
\newtheorem{thm}{Theorem}[section]
\newtheorem{lma}[thm]{Lemma}
\newtheorem{cor}[thm]{Corollary}
\newtheorem{defn}[thm]{Definition}

\newtheorem{prop}[thm]{Proposition}

\newtheorem{ques}[thm]{Question}

\renewcommand{\geq}{\geqslant}
\renewcommand{\leq}{\leqslant}
\renewcommand{\H}{\text{H}}
\renewcommand{\P}{\text{P}}

\DeclareMathOperator{\R}{\mathbb{R}}
\DeclareMathOperator{\Q}{\mathbb{Q}}
\DeclareMathOperator{\N}{\mathbb{N}}
\DeclareMathOperator{\Slice}{Slice}
\DeclareMathOperator{\Fix}{Fix}

\allowdisplaybreaks

\title{The Assouad dimension of randomly generated fractals}

\author{Jonathan M. Fraser$^1$ \and Jun Jie Miao$^2$ \and Sascha Troscheit$^3$}

\begin{document}

\maketitle

\begin{center}

\emph{Dedicated to Kenneth Falconer on the occasion of his 60th birthday.}
\\ \vspace{9mm}

	$^1$School of Mathematics, The University of Manchester, Manchester, M13 9PL, UK.\\
	E-mail: jonathan.fraser@manchester.ac.uk
	\\ \vspace{3mm}
	$^2$Department of Mathematics, Shanghai Key Laboratory of PMMP, East China Normal University, Shanghai 200241, P. R. China.
	\\
	E-mail: jjmiao@math.ecnu.edu.cn
	\\ \vspace{3mm}
	$^3$Mathematical Institute, University of St Andrews, North Haugh, St Andrews, Fife, KY16 9SS, UK.
	\\
	E-mail: s.troscheit@st-andrews.ac.uk
	\end{center}

\begin{abstract}
We consider several different models for generating random fractals including random self-similar sets, random self-affine carpets, and fractal percolation. In each setting we compute either the \emph{almost sure} or the \emph{Baire typical} Assouad dimension and consider some illustrative examples.
Our results reveal a common phenomenon in all of our models: the Assouad dimension of a randomly generated fractal is generically as big as possible and does not depend on the measure theoretic or topological structure of the sample space.
This is in stark contrast to the other commonly studied notions of dimension like the Hausdorff or packing dimension. \\


\end{abstract}

\section{Introduction}

In this paper we study the generic Assouad dimension for a variety of  different models for generating random fractal sets.  The first model is that of  {\it random iterated function systems}.  This well-studied random construction is based on randomising the classical iterated function system in a natural way.  Since this model will be used in several instances throughout the paper we will describe it in detail in Section \ref{randomIntro}.  Firstly, we seek to study the almost sure Assouad dimension of attractors generated in this way, and to do so we specialise to the self-similar setting, in Section \ref{SSSection}, and the setting of self-affine carpets, in Section \ref{BMSection}.  In particular, these two sections seek the generic dimension from a measure theoretic point of view.  Secondly, in Section \ref{typsection} we consider the same model but seek the generic dimension from a topological point of view, i.e. for a residual subset of the sample space. This approach was initiated by the first author in \cite{Fraser12b} and, as in that paper, we are able to compute the generic dimension in a much more general setting than the measure theoretic approach permits. Our second model is {\it fractal percolation}, often referred to as {\it Mandelbrot percolation}, which is also a famous and well-studied model for randomness.  Our results on fractal percolation will be given in the self-contained Section \ref{percSect}. In particular, conditioned on non-extinction, we compute the almost sure Assouad dimension of fractal percolation as well as the almost sure Assouad dimension of all orthogonal projections of the percolation simultaneously.   A somewhat surprising corollary of our results is that, conditioned on non-extinction, almost surely the fractal percolation cannot be embedded in any lower dimensional Euclidean space, no matter how small the almost sure Hausdorff dimension is.  All proofs will be given in Section \ref{proofs}.  Finally, in Section \ref{questionsection} we discuss our results and pose some questions of particular interest to us.
\\ \\
The key common theme throughout the paper and across the somewhat disparate array of questions is that the Assouad dimension is always generically as large as possible.  In the measure theoretic setting this behaviour is completely different from that observed by other important notions of dimension, such as Hausdorff, packing or box counting, where these dimensions are generically an intermediate value, which can take the form of an appropriately weighted average of deterministic values. In order to put our results into context, our results will be compared throughout the paper to those  showing some kind of `averaging'.  In the topological setting, the generic dimensions of random fractals were shown to be `extremal' in \cite{Fraser12b}: some are generically as small as possible and others are generically as large as possible.  The interesting thing for us is that for the Assouad dimension of random attractors the measure theoretic and topological answers agree. This is also in stark contrast with what is `usually' the case.  A classical example being that Lebesgue almost all real numbers are normal, but a residual set of real numbers are as far away from being normal as possible \cite{Hyde10, Salat66}.

\subsection{The Assouad dimension}

The Assouad dimension is the main object of study in this paper and will be defined and discussed in this section.  It was introduced by the French mathematician Patrice Assouad during his PhD research at Universit\'{e} Paris XI, Orsay in the 1970s \cite{assouadphd, Assouad79}.  Assouad's original motivation was to study embedding problems, a subject where the Assouad dimension is still playing a fundamental r\^ole, see \cite{Olson02, Olson10, Robinson11}.  The concept has also found a home in other areas of mathematics, including the theory of quasi-conformal mappings \cite{Heinonen01, Luukkainen98, mackaytyson}, and more recently it is gaining substantial attention in the literature on fractal geometry \cite{Kaenmaki13, Mackay11, Olsen11a, Fraser14a, Li14, Fraser14, Olson14}.  It is also worth noting that, due to its intimate relationship with tangents, it has always been present, although behind the scenes, in the pioneering work of Furstenberg on micro-sets and the related ergodic theory which goes back to the 1960s, see \cite{Furstenberg08}. The Assouad dimension also plays a r\^ole in the fractional Hardy inequality. If the boundary of a domain in $\R^d$ has Assouad dimension less than or equal to $d-p$, then the domain admits the fractional $p$-Hardy inequality~\cite{Aikawa91, Koskela03, Lehrback13}. The Assouad dimension gives a coarse and heavily localised description of how `thick' a given metric space is on small scales; hence its importance for embedding problems.  Most of the other popular notions of dimension, like the Hausdorff, packing, or box dimension, give much more global information, taking an `average thickness' over the whole set.  As such, exploring and understanding the relationships, similarities, and differences, between the Assouad dimension and the other global dimensions is of high and increasing interest, and is one of the key themes of this paper.
\\ \\
Let $(X,d)$ be a metric space and for any non-empty subset $F \subseteq X$ and $r>0$, and let $N_r (F)$ be the smallest number of sets with diameter less than or equal to $r$ required to cover $F$.  The \emph{Assouad dimension} of a non-empty subset $F$ of $X$, $\dim_\text{A} F$, is defined by
\begin{multline*}
\dim_\text{A} F \ = \  \inf \Bigg\{ \alpha \  : \ \text{     there exists constants $C, \, \rho>0$ such that,} \\
\text{ for all $0<r<R\leq \rho$, we have $\ \sup_{x \in F} \, N_r\big( B(x,R) \cap F \big) \ \leq \ C \bigg(\frac{R}{r}\bigg)^\alpha$ } \Bigg\}.
\end{multline*}
It is worth remarking that some authors do not include the $\rho$ in the above definition, i.e. they allow $r$ and $R$ to be arbitrarily large.  The reason for this is to guarantee invariance of the Assouad dimension under specific types of maps, for example the involution $x \mapsto x/\lvert x \rvert^2$ $(x \in (0,1))$, see~\cite[Theorem A.10 (1)]{Luukkainen98}.  This clearly gives rise to a larger quantity, but for bounded sets $F$ the two notions are equivalent and in this paper, as with most papers on fractal geometry, we only consider bounded sets.  We also note that the Assouad dimension can be defined in a number of slightly different ways, but all leading to the same concept.  For example, the function $N_r(F)$ can be replaced by the maximum size of an $r$-packing of the set $F$, or the minimum number of closed cubes of side length $r$ required to cover $F$. For a review of the other notions of dimension, which we frequently mention in this article, but never use directly, like the Hausdorff, packing or box dimension, see~\cite[Chapters 2-3]{FractalGeo}.

\subsection{Random iterated function systems}\label{randomIntro}

Our results in Sections \ref{SSSection}, \ref{BMSection} and \ref{typsection} use the random iterated function systems model, which we recall and discuss in this section.  This model has been used extensively in the literature and fits in to the more general framework of $V$-variable fractals introduced and discussed in detail in~\cite{superfractals, Barnsley05, Barnsley08, Barnsley12}.  Let $(X, d)$ be a compact metric space.
A (deterministic) iterated function system (IFS) is a finite non-empty set of contraction mappings on $X$.
Given such an IFS, $\{S_1, \dots, S_m\}$, it is well-known that there exists a unique non-empty compact set $F$ satisfying
\[
F= \bigcup_{i=1}^{m} S_{i}(F),
\]
which is called the attractor of the IFS, see~\cite{Hutchinson81,FractalGeo}.  We define a random iterated function system (RIFS) to be a set $\mathbb{I} = \{ \mathbb{I}_1, \dots, \mathbb{I}_N\}$, where each $\mathbb{I}_i$ is a deterministic IFS, $\mathbb{I}_i = \{ S_{i,j} \}_{j \in \mathcal{I}_i}$, for a finite index set, $\mathcal{I}_i$, and each map, $S_{i,j}$, is a contracting self-map on $X$.
We define a continuum of attractors of $\mathbb{I}$ in the following way.
Let $\Lambda= \{1, \dots, N\}$, $\Omega = \Lambda^\mathbb{N}$ and let $\omega = (\omega_1, \omega_2, \dots) \in \Omega$.
The attractor of $\mathbb{I}$ corresponding to $\omega$ is defined to be
\begin{equation}\label{attractoreq}
F_\omega \, = \, \bigcap_{k}\  \bigcup_{i_1\in \mathcal{I}_{\omega_1}, \dots, i_k \in\mathcal{I}_{\omega_k}} S_{\omega_1, i_1} \circ \dots \circ S_{\omega_k,i_k}(X).
\end{equation}
So, by `randomly choosing' $\omega \in \Omega$, we `randomly choose' an attractor $F_\omega$.  We now wish to make statements about the generic nature of $F_\omega$.  In particular, our key question is: ``What is the Assouad dimension of $F_\omega$ for generic $\omega \in \Omega$?''  We note that in this paper we adopt two different definitions of `generic'.  Firstly, we may mean \emph{almost surely with respect to a natural probability measure on} $\Omega$ or, secondly, we may mean \emph{for a residual subset of} $\Omega$.  In Sections \ref{SSSection} and \ref{BMSection} we adopt the first approach, which we describe here.  The second, topological approach, will be discussed in Section \ref{typsection}, the only section where it will be used.
Section \ref{percSect} will concern Mandelbrot percolation so here `generic' will again refer to \emph{almost surely with respect to a natural probability measure}, although the measure in that section is different to the one described here.
\\ \\
We may refer to elements in $\Omega$ as \emph{realisations} and, in general, for symbolic codings refer to finite or infinite sequences as words.
For an alphabet $\mathcal A$, we write $\mathcal A^{\mathbb N}$ and $\mathcal A^{*}$ to denote all infinite and finite sequences with entries in $\mathcal A$, respectively.
Given two (finite or infinite) words $v,w$ we write $v\wedge w$ for the maximum number of common initial entries: $v\wedge w=\max\{k\mid v_i=w_i \text{ for }1\leq i\leq k\}$, where we assume for convenience that $\max \varnothing = 0$.
We write $v|_k$ to denote the finite word $u$ of length $k$ such that $v\wedge u=k$.
We also write $\sigma$ to denote the one-sided left shift on $\Omega$, i.e. $\sigma(w)=\sigma(w_1,w_2,\hdots)=(w_2,w_3,\hdots)$.
We define a metric on our space $\Omega$ by $d(x,y)=2^{-(x\wedge y)}$ for $x,y\in \Omega$ with $x \neq y$ and use this metric to define the topology on $\Omega$.  This topology is also generated by the cylinder sets $\{C_k(w)=\{\omega\in\Omega \mid \omega|_k=w\}\mid w\in\Omega, k\geq 1\}$.
Let $\underline p=\{p_1,p_2,\hdots,p_{|\Lambda|}\}$, with $p_{i}>0$ be a finite probability vector and define a Borel probability measure $\mu$ on $\Omega$ by 
\begin{equation}\label{def_m}
\mu(C_k(w))=p_{w_1}p_{w_2}\hdots p_{w_k}.
\end{equation}

This measure is called a \emph{Bernoulli measure} and will be used to describe (measure theoretically) generic properties of $F_\omega$.  There is a large body of literature concerning this model of randomness, often centred on the question of almost sure dimension.  For example, in the self-similar setting see~\cite{Barnsley12} and in the self-affine carpet setting see~\cite{Gui08, Fraser11}.  For $i \in \Lambda$, we write $\underline{i} = ( i, i, \dots,) \in \Omega$ and note that $F_{\underline{i}}$ is the \emph{deterministic} attractor of the IFS $\mathbb{I}_i$.

\section{The self-similar setting}\label{SSSection}

In this section we restrict ourselves to RIFSs where all the mappings in the IFSs that make up the RIFS are similarity mappings on Euclidean space endowed with the usual metric.
That is, for every $S_{i,j}\in\mathbb{I}_{i}\in\mathbb{I}$ and all $x,y\in X$, we have
\begin{equation}\label{def_ss}
\lvert S_{i,j}(x) - S_{i,j}(y) \rvert=c_{i,j}\lvert x - y \rvert,
\end{equation}
where $0<c_{i,j}<1$ is a constant only depending on $i$ and $j$. Without loss of generality, we may assume $X = [0,1]^d$ for some $d \in \mathbb{N}$.  Deterministic self-similar sets (i.e. attractors of a single IFS consisting of maps of this form) are among the most studied examples of fractals in the literature.  Consider for the time being such a deterministic IFS $\{S_i\}_{i \in \mathcal{I}_0}$ with contraction ratios $\{c_i\}_{i \in \mathcal{I}_0}$.  It follows from standard results that the Hausdorff, packing and box dimensions always coincide for self-similar sets, see~\cite[Chapter 3]{TecFracGeo}. It was unknown until recently whether or not the Assouad dimension also always coincides with the other dimensions, but it was proved in~\cite{Fraser14a} that this was not true by providing an explicit counter example and then a general theory was developed in~\cite{Fraser14}.  The key properties which decide if the Assouad and Hausdorff dimensions coincide are various separation conditions which control how the different pieces of the self-similar set overlap with each other.  We recall these now, as they are relevant for this study. The \emph{open set condition} (OSC) was introduced by Moran in~\cite{Moran46} and has played a fundamental r\^ole in the theory of self-similar sets ever since.
\begin{defn}[OSC]
An IFS $\{S_{i}\}_{i\in \mathcal{I}_{0}}$ satisfies the \emph{open set condition} if there exists a bounded, open, non-empty set $U$ such that
\[
\bigcup_{i\in\mathcal{I}_0}S_{i}(U)\subseteq U
\]
with union disjoint.
\end{defn}
If the OSC is satisfied for a self-similar set, then it follows that the Hausdorff and Assouad dimensions coincide and are given by the solution to the following equation:
$$
\sum_{i\in \mathcal{I}_{0}} c_i^s =1.
$$
This equation is known as the \emph{Hutchinson-Moran formula} and the solution is known as the \emph{similarity dimension}.  Even if the OSC is not satisfied, the similarity dimension still provides an upper bound for the Hausdorff dimension, but not for the Assouad dimension, see~\cite[Chapter 9]{FractalGeo} and~\cite{Fraser14a}.  However, it is not the OSC which determines if the Hausdorff and Assouad dimensions coincide, but rather the weak separation property (WSP) introduced by Lau and Ngai~\cite{Lau99} and Zerner~\cite{Zerner96}.  Let
\[
\mathcal{E}=\{S_{\alpha}^{-1}\circ S_{\beta}\mid \alpha,\beta\in\mathcal{I}_{0}^{*}\text{ with }\alpha\neq \beta\}.
\]
\begin{defn}[WSP]
An IFS $\{S_{i}\}_{i\in \mathcal{I}_{0}}$ satisfies the \emph{weak separation property} if
\[
 I \notin \overline{\mathcal{E} \setminus\{I\}},
\]
where $I$ is the identity map.
\end{defn}

In other words the weak separation property is satisfied if the identity is not an accumulation point of $\mathcal E\setminus\{I\}$, that is there is no sequence of pairs $(\{\alpha_k,\beta_k\})_{k=1}^{\infty}$ such that $S_{\alpha_k}^{-1}\circ S_{\beta_k}\to I$ but $S_{\alpha_k}^{-1}\circ S_{\beta_k}\neq I$ for all $k$. Note that every IFS satisfying the OSC, trivially satisfies the WSP but there are examples of IFSs that satisfy the WSP but not the OSC. We make use of such an example in Section~\ref{counterexample}.  It is worth noting that the OSC can also be defined via the set $\mathcal{E}$.  In particular, combining work of Bandt and Graf~\cite{Bandt92}, and Schief~\cite{Schief94}, one obtains that an IFS of similarities satisfies the OSC if and only if $ I \notin \overline{\mathcal{E}}$.
\\ \\
In order to state some of our results, we are required to use a random analogue of the OSC.
\begin{defn}[UOSC]
We say that $\mathbb{I}$ satisfies the \emph{uniform open set condition (UOSC)}, if each deterministic IFS $\mathbb{I}_{i}$ satisfies the OSC and the open set can be chosen uniformly, i.e., there exists a non-empty open set $U \subseteq X$ such that for each $i \in \Lambda$ we have
\[
\bigcup_{j \in \mathcal{I}_i} S_{i,j}(U) \subseteq U
\]
with the union disjoint.
\end{defn}

This separation condition is natural in the random setting and was used in the papers~\cite{Barnsley05, Barnsley12, Fraser12b}, for example.

\subsection{Our results on random self-similar sets}

First we obtain a \emph{sure} upper bound, i.e. an upper bound which holds for \emph{all} realisations.

\begin{thm}\label{theo2-1}
Let $\mathbb{I}$ be a RIFS consisting of IFSs of similarities as in $(\ref{def_ss})$.  Assume that $\mathbb{I}$ satisfies the UOSC.  Then, for all $\omega \in \Omega$, we have
\[
\dim_{\text{\emph{A}}}F_{\omega} \ \leq \ \max_{i\in \Lambda} \, \dim_{\text{\emph{A}}}F_{\underline i}.
\]
\end{thm}
The proof of Theorem \ref{theo2-1} will be given in Section ~\ref{theo2-1proof}.  Note that for each $i \in \Lambda$, $\dim_{\text{A}}F_{\underline i}$ is the Assouad dimension of the deterministic self-similar set $F_{\underline i}$, which may be computed via the Hutchinson-Moran formula since the OSC is satisfied. We will provide an example in Section \ref{counterexample} showing that this upper bound can fail if we do not assume the UOSC.  We are also able to obtain an almost sure lower bound.
\begin{thm}\label{theo2-2}
Let $\mathbb{I}$ be a RIFS consisting of IFSs of similarities as defined in $(\ref{def_ss})$.  Then, for almost all $\omega \in \Omega$, we have
\begin{equation}\label{theo2-2eq}
\dim_{\text{\emph{A}}}F_{\omega} \ \geq \ \max_{i\in \Lambda}  \, \dim_{\text{\emph{A}}}F_{\underline i}.
\end{equation}
\end{thm}
The proof of Theorem \ref{theo2-2} will be given in Section~\ref{theo2-2proof}.  Note that the Theorem~\ref{theo2-2} requires no separation conditions, whereas Theorem~\ref{theo2-1} requires the UOSC. Combining the upper and lower estimates immediately yields our main result on random self-similar sets.
\begin{thm} \label{SSresult}
Let $\mathbb{I}$ be a RIFS consisting of IFSs of similarities as defined in $(\ref{def_ss})$.  Assume that $\mathbb{I}$ satisfies the UOSC.  Then
\[
\dim_{\text{\emph{A}}}F_{\omega} \ = \ \max_{i\in \Lambda} \, \dim_{\text{\emph{A}}}F_{\underline i},
\]
 for almost all $\omega\in\Omega$.
\end{thm}
The results above are in stark contrast to the analogous almost sure formulae for the Hausdorff, packing and box dimension which are some form of weighted average of the deterministic values.  One can, for example, show that the Hausdorff dimension of a random 1-variable self-similar set satisfying the UOSC is almost surely given by the unique zero of the weighted average of the logarithm of the Hutchinson-Moran formulae for the individual IFSs~\cite[Remark 4.5]{Barnsley12}, i.e the unique $s$ satisfying
\[
\sum_{i\in\Lambda}p_{i}\log\left(\sum_{j\in\Lambda_{i}}c_{i,j}^{s}\right)=0.
\]
 A neat consequence of this is that the Assouad dimension and the Hausdorff dimension can be almost surely distinct, no matter which separation condition you assume.  Recall that in the deterministic setting the WSP is sufficient to guarantee equality, and in the random setting it was proved by Liu and Wu that the Hausdorff and box dimensions almost surely coincide, even if there are overlaps \cite{Liu02}.  In fact the only way the Assouad and Hausdorff dimensions can almost surely coincide in the UOSC case is if all of the deterministic IFSs had the same similarity dimension.  Also, apart from in this rare situation, our result shows that random self-similar sets are almost surely not Ahlfors regular, as for Ahlfors regular sets the Hausdorff and Assouad dimensions coincide.  Finally we obtain precise information on the size of the exceptional set in Theorem \ref{SSresult}.
\begin{thm}\label{exceptionalResult}
Let $\mathbb{I}$ be a RIFS consisting of IFSs of similarities as defined in $(\ref{def_ss})$. Assume that $\mathbb{I}$ satisfies the UOSC and that $\dim_{\text{\emph{A}}}F_{\underline i}$ is not the same for all $i \in \Lambda$, i.e. the similarity dimensions of the deterministic attractors are not all the same.  Then the exceptional set
\[
E \ = \ \big\{ \omega \in \Omega \mid \dim_{\text{\emph{A}}}F_{\omega} < \max_{i\in \Lambda} \, \dim_{\text{\emph{A}}}F_{\underline i} \big\}
\]
is a set of full Hausdorff dimension, despite being a $\mu$-null set, i.e. $\dim_{\text{\emph{H}}}E=\dim_{\text{\emph{H}}}\Omega$.
\end{thm}

The proof of Theorem \ref{exceptionalResult} can be found in Section~\ref{exceptionalProof2}. The following two figures depict some examples of random self-similar sets.  The RIFS is made up of three deterministic IFSs, which are shown in Figure \ref{picdeterministicSS}.  Dotted squares indicate the (homothetic) similarities used.  In Figure \ref{picrandomSS}, three different random realisations are shown, which will (almost surely) all have the same Assouad dimension as the maximum of the three deterministic values.

\begin{figure}[H]
\setlength{\unitlength}{\textwidth}
\begin{picture}(1,0.28)
\put(0,0){\includegraphics[width=0.28\unitlength]{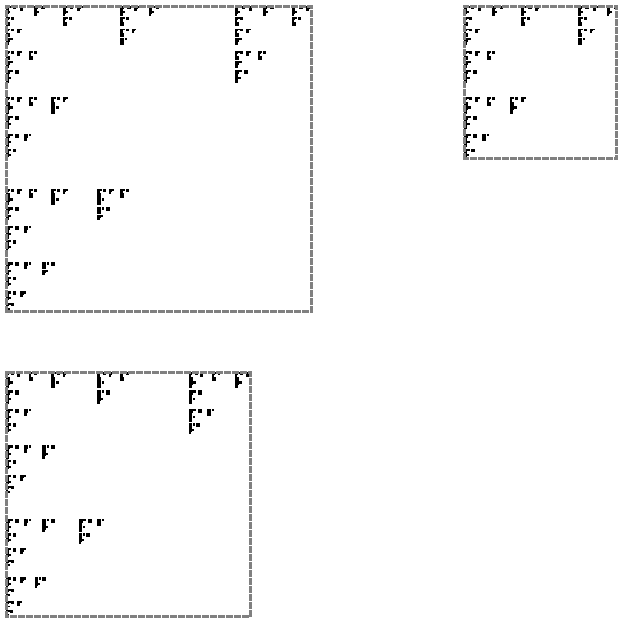}}
\put(0.35,0){\includegraphics[width=0.28\unitlength]{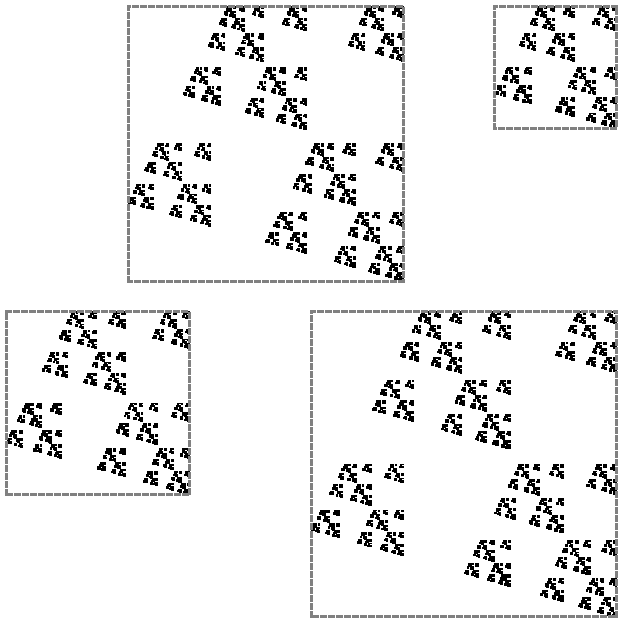}}
\put(0.7,0){\includegraphics[width=0.28\unitlength]{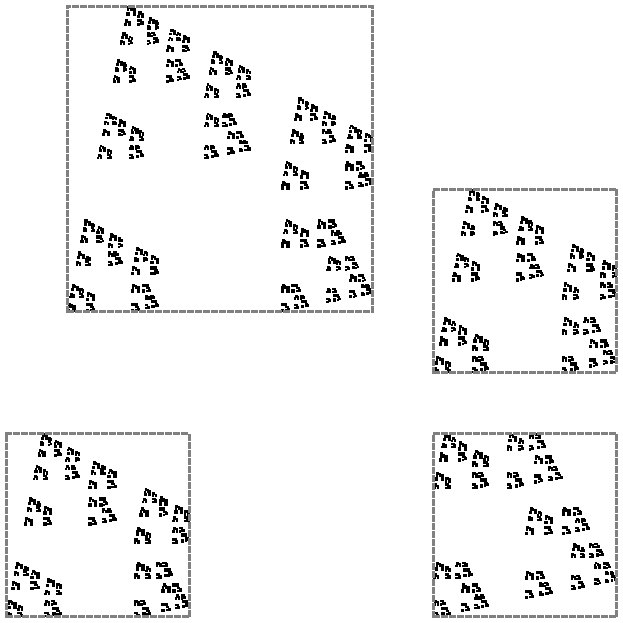}}
\end{picture}
\caption{Deterministic self-similar attractors $F_{\underline 1}$, $F_{\underline 2}$ and $F_{\underline 3}$.}
\label{picdeterministicSS}
\end{figure}

\begin{figure}[H]
\setlength{\unitlength}{\textwidth}
\begin{picture}(1,0.28)
\put(0,0){\includegraphics[width=0.28\unitlength]{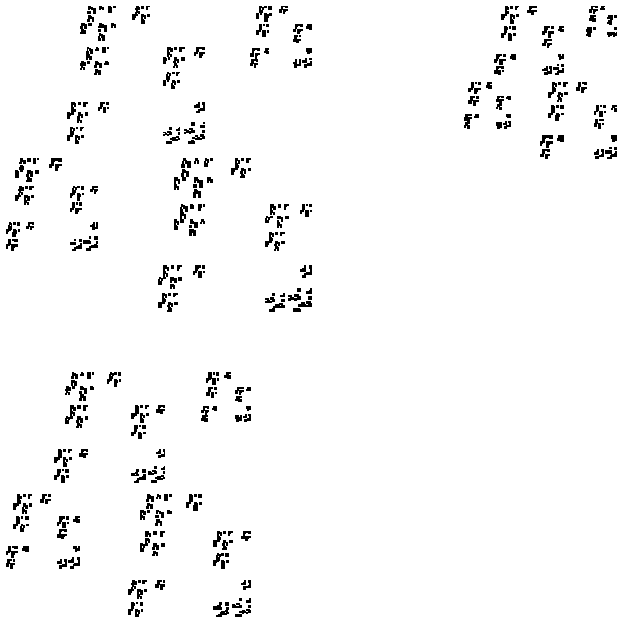}}
\put(0.35,0){\includegraphics[width=0.28\unitlength]{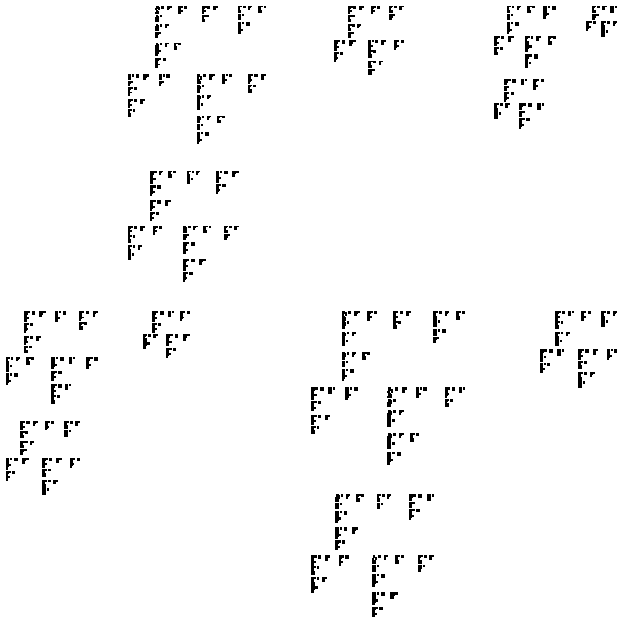}}
\put(0.7,0){\includegraphics[width=0.28\unitlength]{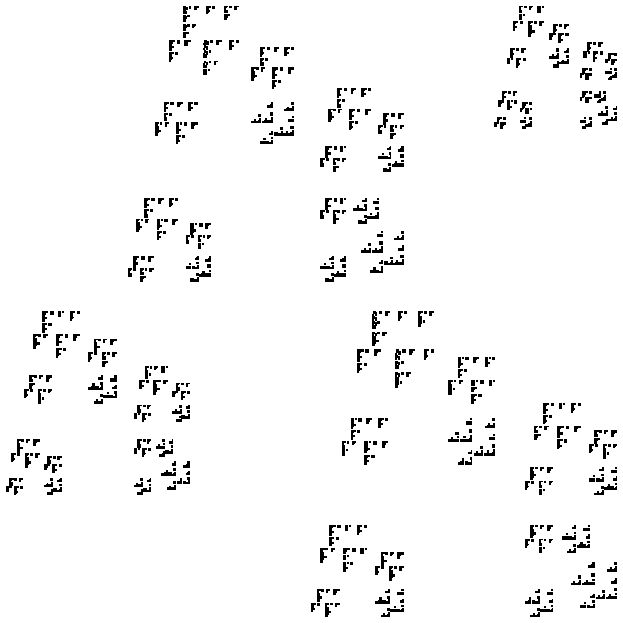}}
\end{picture}
\caption{Random self-similar attractors $F_{\alpha}$, $F_{\beta}$ and $F_{\gamma}$ for different realisations $\alpha=(1,2,3,1,2,1,3,3,\hdots),\,  \beta=(2,1,2,1,1,1,1,3,\hdots), \, \gamma=(2,3,3,2,1,1,1,3,\hdots) \in\Omega$.}
\label{picrandomSS}
\end{figure}

We finish this section by mentioning that Li, Li, Miao and Xi \cite{Li14} studied the Assouad dimension of Moran sets $E$ generated by two sequences $\{n_k\in\mathbb{N}\}_{k=1}^\infty$ and $\{\phi_k\in\mathbb{R}^{n_k}\}_{k=1}^\infty$, where $n_k$ indicates the number of contractions, and $\phi_k=(c_{k,1},\cdots,c_{k,n_k})$ gives the contraction ratios at the $k$th level. They show that $\displaystyle \dim_\text{A} E =\lim_{m\to\infty }  \sup_{k} s_{k,k+m}$, where $s_{k,k+m}$ is the unique solution  to the equation
$$
\prod_{i=k+1}^{k+m}  \sum_{j=1}^{n_i}( c_{i,j})^s=1.
$$
By choosing $\phi_k=(c_{k,1},\cdots,c_{k,n_k})$ from a fixed number of patterns, such a Moran set may be regarded as a particular realisation of our random self-similar sets.  Therefore this result gives information about specific realisations, whereas our results study the generic situation.

\subsection{An example with overlaps}\label{counterexample}

Here we provide an example showing that the assumption of some separation condition in Theorem \ref{theo2-1} is necessary.
Let the RIFS $\mathbb{I}$ be the system consisting of two IFSs of similarities, $\mathbb{I}_{1}$ and $\mathbb{I}_{2}$.
Let $\mathbb{I}_{1}$ be the IFS consisting of the three maps $S_{1,1},S_{1,2}$ and $S_{1,3}$ and $\mathbb{I}_{2}$ consist of the three maps $S_{2,1},S_{2,2}$ and $S_{2,3}$, where $S_{i,j}:\R\to\R$ and
\begin{align*}
S_{1,1}&=\tfrac{1}{2}x,&
S_{1,2}&=\tfrac{1}{4}x,&
S_{1,3}&=\tfrac{1}{16}x+\tfrac{15}{16},\\
S_{2,1}&=\tfrac{1}{3}x,&
S_{2,2}&=\tfrac{1}{9}x,&
S_{2,3}&=\tfrac{1}{81}x+\tfrac{80}{81}.
\end{align*}
As $S_{i,1}$ and $S_{i,2}$ have the same fixed point for $i =1,2$,  both $\mathbb{I}_{1}$ and $\mathbb{I}_{2}$ fail the OSC. This can be shown by taking $\alpha=(1,1)$, $\beta=(2)$. We have $S^{-1}_{i,\alpha}\circ S_{i,\beta}\in\mathcal E$ but $S^{-1}_{i,\alpha}\circ S_{i,\beta}=I$ and so $I\in\mathcal E$, which means the two IFSs fail the OSC and so $\mathbb{I}$ fails to satisfy the UOSC.
Furthermore, if one considers the individual IFSs, since $\frac{\log c_{i,1}}{\log c_{i,2}}\in\Q$ for $i=1,2$, one can show directly from the definition that the WSP is satisfied.  Therefore, for both systems the Assouad and Hausdorff dimensions coincide and are therefore no greater than their similarity dimensions, see \cite{Fraser14}.  That is $\dim_{\text{A}}F_{\underline i}\leq s_{i}$ for $i=1,2$, where $s_{i}$ is given implicitly by the Hutchinson-Moran formula $\sum_{j=1}^{3}c_{i,j}^{s_{i}}=1$.
Solving numerically we find that $s_{1}\approx0.81137$ and $s_{2}\approx0.511918$ and so $\max_{i\in \Lambda}\, \dim_{\text{A}}F_{\underline i} <1$.
Consider however $\omega=(1,2,1,2,1,\hdots)$.
This is equivalent to the deterministic IFS consisting of the 9 possible compositions of a map from $\mathbb{I}_{1}$ with a map from $\mathbb{I}_{2}$. Consider just the two maps
\begin{eqnarray*}
T_{1}&=&S_{1,1}\circ S_{2,2}=\tfrac{1}{18}x,\\
T_{2}&=&S_{1,2}\circ S_{2,1}=\tfrac{1}{12}x.
\end{eqnarray*}
One can check that $\log18/\log12\notin\Q$ and therefore using an argument similar to the one in \cite[Section~3.1]{Fraser14a} one can show that $\dim_{\text{A}}F_{(1,2,1,\hdots)}=1$, which is strictly greater than the maximum given by the deterministic IFS, showing that if the UOSC is not satisfied, then the Assouad dimension of particular realisations can exceed the maximum of the deterministic values.

\begin{figure}[H]
\centering
\setlength{\unitlength}{\textwidth}
\includegraphics[width=0.8\unitlength]{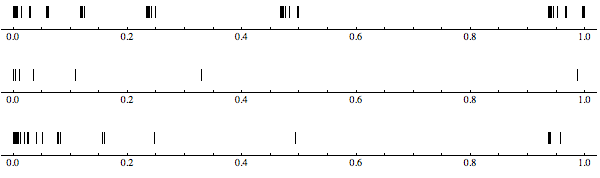}
\caption{Top and middle: the two deterministic attractors which both have Assouad dimension strictly smaller than 1.  Bottom: the random self-similar set for the realisation $\omega = (1,2,1,2,\dots)$ which has Assouad dimension 1.  Stretching the imagination slightly, one can see the unit interval emerging as a tangent at the origin for the third set, but the rational dependence between the contraction ratios prevents this happening for the first two examples.}
\label{picrandomSSoverlaps}
\end{figure}

\section{Almost sure dimension for random self-affine carpets}\label{BMSection}

Self-affine carpets are a special type of planar self-affine set that were first studied independently by Bedford and McMullen in the 1980s \cite{Bedford84,McMullen84}.
The properties of Bedford and McMullen's original model and various generalisations by Bara\'nski \cite{Baranski07}, Lalley and Gatzouras \cite{LalleyGatzouras92}, Feng and Wang \cite{Feng05}, and the first author\cite{Fraser12} have been extensively studied with the main aim to find the Hausdorff, packing and box-counting dimension.  Several random versions have also been considered including 1-variable randomisations \cite{Gui08, Fraser11} and statistically self-affine constructions \cite{LalleyGatzouras94}. More recently, some authors have also considered the Assouad dimension of these sets, see \cite{Mackay11,Fraser14a}.  Fraser and Shmerkin \cite{FraserShmerkin} considered the dimensions of random self-affine carpets where for them the randomness was obtained by randomly translating the column structure.  They computed the almost sure Hausdorff and box dimensions and remarked that the situation for the Assouad dimension was not clear because the Assouad dimension could `jump up' above the initial value.  It turns out that in our model, the Assouad dimension is similarly sensitive to `jumping up' and we show that, in a different context, the Assouad dimension of random self-affine carpets can again `jump up' above the initial and expected values, see the example in Section~\ref{carpetscounterex}.
\\ \\
We introduce Bedford-McMullen carpets here with the IFS and RIFS notation introduced in Section \ref{randomIntro}. For each $i \in \Lambda$, let $m_i,n_i$ be fixed integers with $n_i > m_i \geq 2$. Then, for each $i \in \Lambda$, divide the unit square $[0,1]^2$ into a uniform $m_i \times n_i$ grid and select a subset of the sub-rectangles formed. Let the IFS $\mathbb{I}_i$ be made up of the affine maps which take the unit square onto each chosen sub-rectangle without any rotation or reflection.  As such the constituent maps $S_{i,j}:[0,1]^2 \to [0,1]^2$ are of the form
\[
S_{i,j}=\left(\frac{x}{m_i}+\frac{a_{i,j}}{m_i},\;\frac{y}{n_i}+\frac{b_{i,j}}{n_i}\right),
\]
for integers $ a_{i,j}, b_{i,j}$, where $0\leq a_{i,j}<m_{i}$, and $0\leq b_{i,j}<n_i$. For each $i \in \Lambda$, let $A_i$ be the number of distinct integers $a_{i,j}$ used for maps in $\mathbb{I}_i$, i.e. the number of non-empty columns in the defining pattern for the $i$th IFS.  Also, for each $i \in \Lambda$, let $B_{i} = \max_{k \in \{0, \dots, m-1\}} \lvert \{ S_{i,j} \in \mathbb{I}_i : a_{i,j} = k\}\rvert$, i.e. the maximum number of rectangles chosen in a particular column of the defining pattern for the $i$th IFS.  For the deterministic IFS $\mathbb{I}_i$ with attractor $F_{\underline{i}}$, it was shown by Mackay \cite{Mackay11} that
\begin{equation}\label{MackayAssouadDim}
\dim_\text{A} F_{\underline{i}} \ = \ \frac{\log A_i}{\log m_i} \, + \, \frac{\log B_i}{\log n_i}.
\end{equation}
One interpretation of this is that the Assouad dimension is the dimension of the projection of $F_{\underline{i}}$ onto the first coordinate plus the maximal dimension of a vertical slice through $F_{\underline{i}}$.  A reasonable first guess for the almost sure Assouad dimension of the random attractors of $\mathbb{I}$ would be to take the maximum of equation~(\ref{MackayAssouadDim}) over all $i \in \Lambda$.  Surprisingly this is not the correct answer, as we shall see in this section.  First we prove a sure upper bound, which at first sight does not look particularly sharp.
\begin{thm}\label{upperboundcarpets}
Let $\mathbb {I}$ be as above. Then for all $\omega\in\Omega$
\[
\dim_\text{\emph{A}} F_\omega \  \leq  \ \max_{i\in\Lambda} \frac{\log A_i}{\log m_i} \, + \, \max_{i\in\Lambda} \frac{\log B_i}{\log n_i}.
\]
\end{thm}

Theorem \ref{upperboundcarpets} will be proved in Section \ref{upperboundcarpetsproof}. It turns out that this upper bound is almost surely sharp and this is the content of our main result in the self-affine setting.

\begin{thm}\label{lowerboundcarpets}
Let $\mathbb {I}$ be as above.  Then for almost all $\omega \in \Omega$, we have
\[
\dim_\text{\emph{A}} F_\omega \  =  \  \max_{i\in\Lambda} \, \frac{\log A_i}{\log m_i} \, + \, \max_{i\in\Lambda} \,  \frac{\log B_i}{\log n_i}.
\]
\end{thm}

Theorem \ref{lowerboundcarpets} will be proved in Section \ref{lowerboundcarpetsproof}. As remarked above this is in stark contrast to results concerning the classical dimensions of attractors of RIFS, but still in keeping with the `almost surely maximal' philosophy.
An example of the `averaging' that happens with Hausdorff, packing and box counting dimension is the result by Gui and Li~\cite{Gui08}, where if $m_i=m<n = n_i$ for all $i \in \Lambda$, the almost sure dimension is given by the weighted average of the dimensions of the individual attractors:
\[
\dim F_\omega=\sum_{i\in\Lambda}p_i \dim F_{\underline i},
\]
where $\dim$ can refer to any of the Hausdorff, packing or box counting dimension.
  The key difference between the case considered here and the self-similar case is that, despite whatever separation conditions one wishes to impose, the `maximal value' is not  generally the maximum of the deterministic values.  We construct a very simple example to illustrate this difference in Section \ref{carpetscounterex}.
\\

The following two figures depict some examples of random self-affine Bedford-McMullen carpets.  The RIFS is made up of three deterministic IFSs, which are shown in Figure \ref{BMDeterministic}.  We chose $m_1 = 2$, $n_1=3$, $m_2 = 3$, $n_2=5$, $m_3 = 2$ and $n_3=4$ and indicate the chosen affine maps with rectangles.  In Figure \ref{BMRandom}, three different random realisations are shown.

\begin{figure}[H]
\setlength{\unitlength}{\textwidth}
\begin{picture}(1,0.28)
\put(0,0){\includegraphics[width=0.28\unitlength]{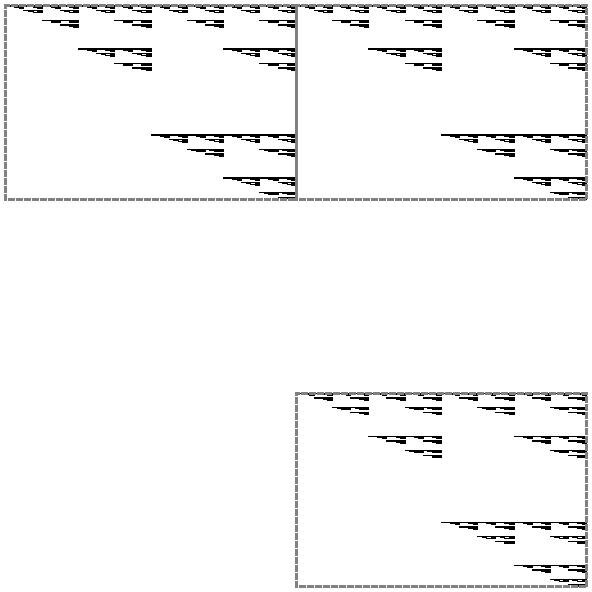}}
\put(0.35,0){\includegraphics[width=0.28\unitlength]{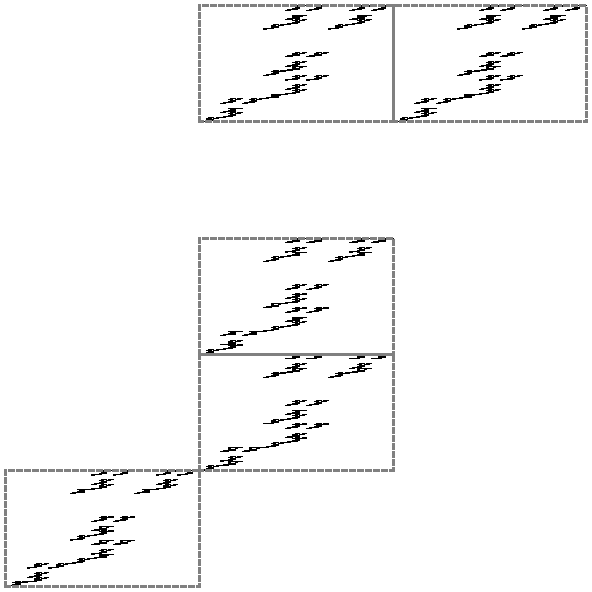}}
\put(0.7,0){\includegraphics[width=0.28\unitlength]{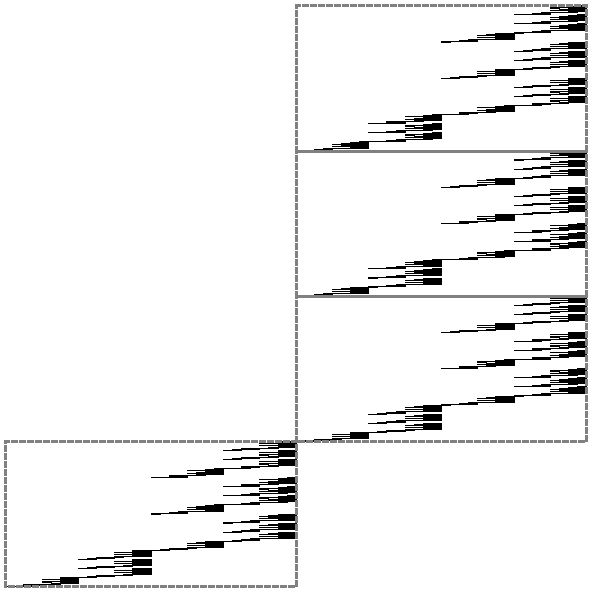}}
\end{picture}
\caption{Deterministic Bedford-McMullen Carpets $F_{\underline 1}$, $F_{\underline 2}$ and $F_{\underline 3}$.}
\label{BMDeterministic}
\end{figure}

\begin{figure}[H]
\setlength{\unitlength}{\textwidth}
\begin{picture}(1,0.28)
\put(0,0){\includegraphics[width=0.28\unitlength]{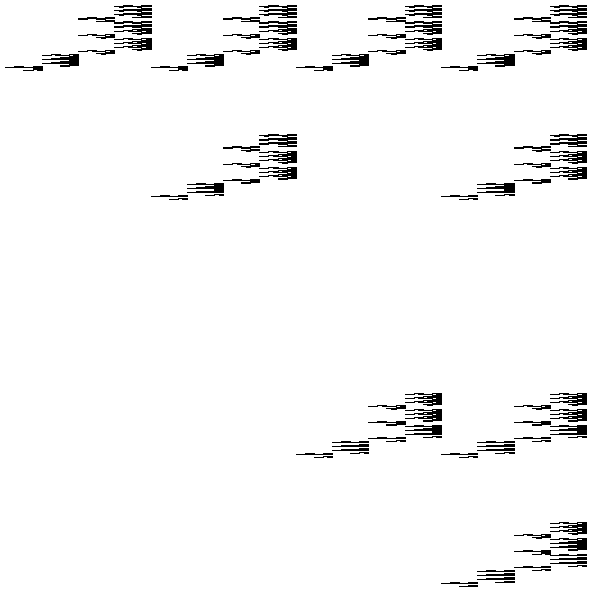}}
\put(0.35,0){\includegraphics[width=0.28\unitlength]{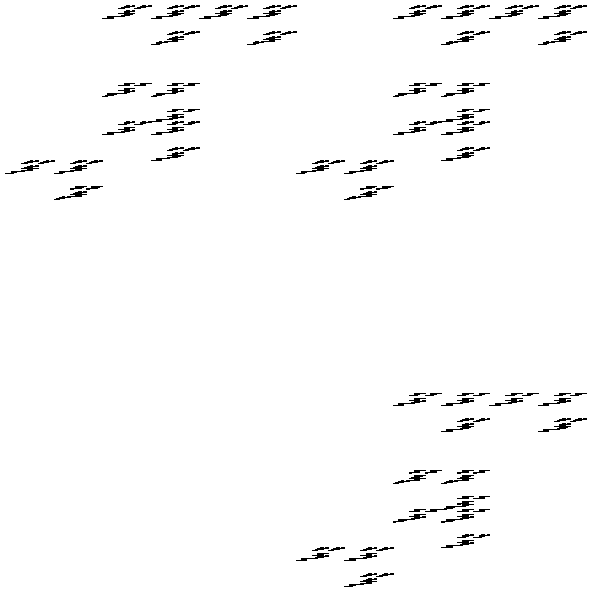}}
\put(0.7,0){\includegraphics[width=0.28\unitlength]{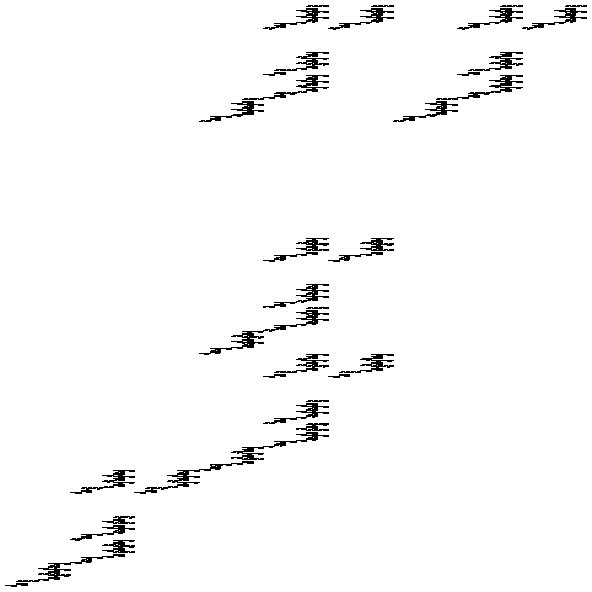}}
\end{picture}
\caption{Random Bedford-McMullen Carpets $F_{\alpha}$, $F_{\beta}$ and $F_{\gamma}$ for realisations $\alpha=(1,1,3,3,1,3,1,3,\hdots), \, \beta=(1,2,1,2,2,3,2,1,\hdots), \, \gamma=(2,2,3,2,1,2,2,2,\hdots) \in\Omega$.}
\label{BMRandom}
\end{figure}

\subsection{An example with larger Assouad dimension than expected}\label{carpetscounterex}

In this section we briefly elaborate on our belief that the formula for the almost sure Assouad dimension of random self-affine carpets returns a  surprisingly large value.  Consider the following very simple example.  Let $\mathbb{I} = \{\mathbb{I}_1, \mathbb{I}_2\}$, where $m=2$ and $n=3$ for both deterministic IFSs.  Let $\mathbb{I}_1$ consist of the maps corresponding to the two rectangles in the top row of the defining grid and let $\mathbb{I}_2$ consist of the maps corresponding to the three rectangles in the right hand column of the defining grid.  Both the deterministic attractors are not very interesting; they are both line segments.  In particular, they both have Assouad dimension equal to 1.  Moreover, it is a short calculation to show that the Assouad dimension of $F_\omega$ is no larger than 1 for any eventually periodic word $\omega$.  This means that, unlike the self-similar example in Section \ref{counterexample}, the Assouad dimension cannot increase by taking a finite combination of the initial IFSs.  However, Theorem \ref{lowerboundcarpets} shows that the Assouad dimension of $F_\omega$ is almost surely 2.

\begin{figure}[H]
\setlength{\unitlength}{\textwidth}
\centering
\includegraphics[width=1\unitlength]{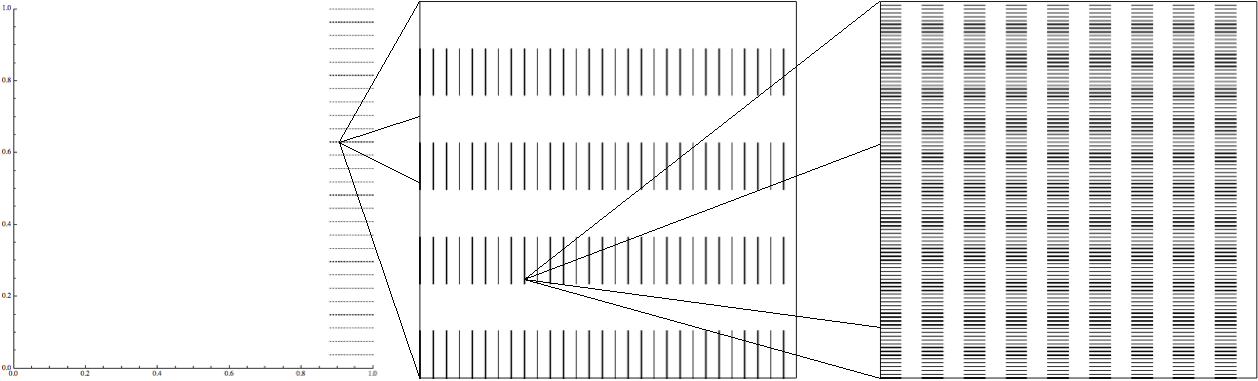}
\caption{The left most image is the random self-affine carpet associated to the above RIFS for the realisation $(3,3,3,1,1,1,1,3, \dots)$.  The other images show small parts of the set blown back up to the unit square.  One can see the zoomed in images filling up more and more space, leading to the unit square being a \emph{very weak tangent} to the random self-affine set.  This is what causes the Assouad dimension to be maximal, see Section \ref{lowerboundcarpetsproof}. }
\end{figure}

\section{Typical Assouad dimension for random attractors} \label{typsection}

In this section we consider an alternative approach to deciding the `generic properties' of random fractals.
This approach is topological rather than measure theoretic and was first considered by the first author \cite{Fraser12b}.
Let $(Y,d_Y)$ be a complete metric space.
A set $N \subseteq Y$ is \emph{nowhere dense} if for all $y \in N$ and for all $r>0$ there exists a point $x \in Y \setminus N$ and $t>0$ such that $B(x,t) \subseteq B(y,r) \setminus N.$  A set $M$ is said to be \emph{of the first category}, or, \emph{meagre}, if it can be written as a countable union of nowhere dense sets.
We think of a meagre set as being \emph{small} and the complement of a meagre set as being \emph{big}.
A set $T \subseteq Y$ is \emph{residual} or \emph{co-meagre}, if $Y \setminus T$ is meagre.
A property is called \emph{typical} if the set of points which have the property is residual.  In many ways a residual set behaves a lot like a set of full measure.  For example, the intersection of a countable number of residual sets is residual and the space cannot be broken up into the disjoint union of two sets which are both residual. As such it is a reasonable replacement for the notion of \emph{almost all} in describing generic properties in a complete metric space.  In Section \ref{proofs} we will use the following theorem to test for typicality without mentioning it explicitly.

\begin{thm}
	In a complete metric space, a set $T$ is residual if and only if $T$ contains a countable intersection of open dense sets or, equivalently, $T$ contains a dense $\mathcal{G}_\delta$ subset of $Y$.
\end{thm}

For a proof of this result and for a more detailed account of Baire Category the reader is referred to \cite{oxtoby}.
By applying these notions to the complete metric space $(\Omega, d)$ we can replace ``full measure'' with ``residual'' to gain our new notion of genericity.
In \cite{Fraser12b} it was shown that these two approaches differ immensely in the context of Hausdorff and packing dimension.
Indeed, it was shown that there exists a residual set $R\subseteq \Omega$ such that for all $\omega\in R$
\[
\dim_\H F_\omega \ = \ \inf_{u \in \Omega} \dim_\H F_u,
\]
and
\[
\dim_\P F_\omega \ = \ \sup_{u \in \Omega} \dim_\P F_u,
\]
for any RIFS consisting of bi-Lipschitz contractions without assuming any separation conditions.
This is very different from the measure theoretic approach, which tends to favour convergence rather than divergence, with the almost sure packing and Hausdorff dimensions often equal to some sort of average over the parameter space, rather than opposite extremes. Our main result in this section proves an analogous result for Assouad dimension. In the wider context of the paper, the main interest of this result is that in the setting of Assouad dimension, the topological and measure theoretic approaches seem to agree. Observe that we are able to compute the \emph{typical} Assouad dimension in a much more general context than the \emph{almost sure} Assouad dimension, but this is not surprising in view of \cite{Fraser12b}.

\begin{thm} \label{baireassouad}
	Let $\mathbb{I}$ be an RIFS consisting of deterministic IFSs of bi-Lipschitz contractions. Then there exists a residual set $R\subseteq \Omega$ such that for all $\omega\in R$
	\[
	\dim_\text{\emph{A}} F_\omega \ = \ \sup_{u \in \Omega} \,  \dim_\text{\emph{A}} F_u.
	\]
\end{thm}

We will prove Theorem \ref{baireassouad} in Section \ref{baireassouadproof}.
Notice that the above result assumes no separation properties and the mappings can be much more general than similarities or even affine maps.
\\ \\
An immediate and perhaps surprising corollary of this is that Theorems \ref{SSresult} and \ref{lowerboundcarpets} remain true even if the measure theoretic approach is replaced by the topological approach adopted in this section.

\section{Fractal Percolation}\label{percSect}

Fractal percolation or Mandelbrot percolation, first appearing in the works of Mandelbrot in the 1970s as a model for intermittent turbulence \cite{Mandelbrot74}, is one of the most well studied and famous examples of a random fractal and is defined as follows.
Begin with the unit cube $Q = [0,1]^d$ in $\mathbb{R}^d$ and fix an integer $n\geq 2$ and a probability $p \in (0,1)$.
Divide the unit cube into a mesh of $n^d$ smaller compact cubes each having side lengths $1/n$.
Now choose to keep each smaller square independently with probability $p$.
The result is a compact collection of cubes, which we call $Q_1$.
Now repeat this process independently with each surviving cube from the first iteration to form another collection of cubes this time of side lengths $1/n^2$, which we denote by $Q_2$.
Repeating this process infinitely many times gives a decreasing sequence of compact collections of increasingly smaller cubes, $Q_k$.
The resulting random set, or fractal percolation, is then defined as
\[
F \ =  \ \bigcap_{k \in \mathbb{N}} Q_k.
\]

This construction has been studied intensively over the last 40 years, with many interesting phenomena being observed.  Initially, most work concerned the classical question of `percolation', namely, is there a positive probability that one face of $Q$ is connected by $F$ to the opposite face?  More recently, a lot of work has been done on generic dimensional properties of $F$, orthogonal (and other) projections of $F$, and slices of $F$.  Rather than cite many papers we simply refer the reader to the recent survey by Rams and Simon \cite{Rams14}.  Concerning the dimension of $F$, if $p>1/n^d$ then there is a positive probability that $F$ is nonempty and conditioned on this occurring, the Hausdorff and packing dimension of $F$ are almost surely given by $\log n^d p / \log n$. Recently, there has also been a lot of work on almost sure properties of the orthogonal projections of fractal percolation. In particular, one wants to obtain a `Marstrand type result' for \emph{all} projections $\pi \in \Pi_{d,k}$ rather than just \emph{almost all}.  Here $\Pi_{d,k}$ is the Grassmannian manifold consisting of all orthogonal projections from $\mathbb{R}^d$ to $\mathbb{R}^k$ ($k \leq d$) identified with $k$ dimensional subspaces of $\mathbb{R}^d$ in the natural way and equipped with the usual Grassmann measure.  Our main result is the following, which gives, conditioned on non-extinction, the almost sure Assouad dimension of $F$ as well as an optimal projection result.

\begin{thm} \label{percolationassouad}
	Let $p>1/n^d$. Then, conditioned on $F$ being nonempty, we have that almost surely
	\[
	\dim_\text{\emph{A}} F = d
	\]
and for all $k \leq d$ and $\pi \in \Pi_{d,k}$ simultaneously,
	\[
	\dim_\text{\emph{A}} \pi F = k.
	\]
\end{thm}

We will prove Theorem \ref{percolationassouad} in Section \ref{percolationassouadproof}.
Observe that, provided $p>1/n^d$, the almost sure Assouad dimension does not depend on $p$.
This is in stark contrast to the Hausdorff and packing dimension case, but by now not surprising to us.  An immediate corollary of Theorem \ref{percolationassouad} is the following embedding theorem for Mandelbrot percolation.
\begin{cor}
	Let $p>1/n^d$. Then, conditioned on $F$ being nonempty, almost surely $F$ cannot be embedded in any lower dimensional Euclidean space via a bi-Lipschitz map, i.e., there does not exists a bi-Lipschitz map $\phi:F \to \mathbb{R}^{d-1}$.
\end{cor}
This follows from Theorem \ref{percolationassouad} and the fact that bi-Lipschitz maps cannot decrease Assouad dimension \cite[Theorem A.5.1]{Luukkainen98}.  This result is somewhat surprising in that given any $\varepsilon>0$, one can choose $p$ sufficiently close to (but greater than) $n^{-d}$, such that almost surely (conditioned on non-extinction) the Hausdorff dimension of $F$ is smaller than $\varepsilon$, but yet $F$ still cannot be embedded in any Euclidean space with dimension less than that of the initial ambient space.

\begin{figure}[H]
\setlength{\unitlength}{\textwidth}
\begin{picture}(1,0.35)
\put(0.1,0){\includegraphics[width=0.35\unitlength]{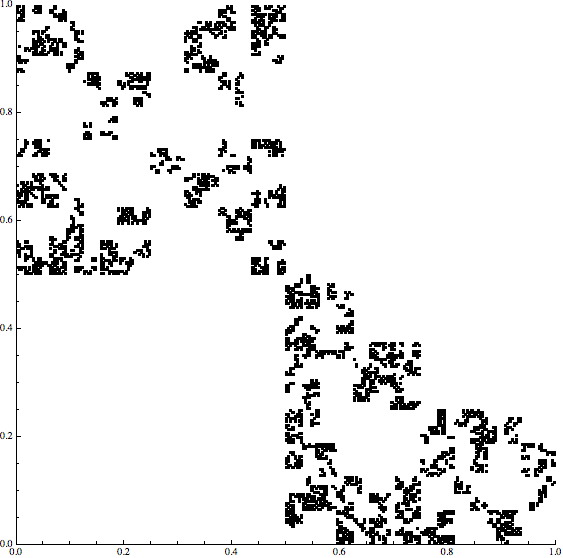}}
\put(0.55,0){\includegraphics[width=0.35\unitlength]{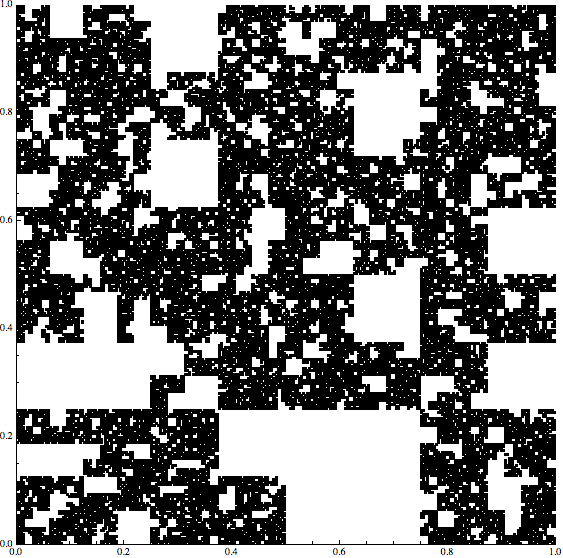}}
\end{picture}
\caption{Percolation limit sets for $p=0.7$ and $p=0.9$ ($n=2$, $d=2$).}
\label{Percolation}
\end{figure}

\newpage

\section{Proofs}\label{proofs}

\subsection{Weak tangents}

One of the most important and convenient ways to estimate the Assouad dimension of a given set from below is to construct \emph{weak tangents}.
This approach was introduced and studied to great effect by Mackay and Tyson \cite{mackaytyson}, and we will utilise it again here.  Let $d_\mathcal{H}$ denote the Hausdorff metric on the space of compact subsets of $\mathbb{R}^d$, defined by
\[
d_\mathcal{H}(A,B) \ = \  \inf \big\{ \varepsilon \geq 0: A \subseteq [B]_{\varepsilon} \text{ and } B \subseteq [A]_{\varepsilon} \big\}
\]
where $[A]_{\varepsilon}$ is the closed $\varepsilon$-neighbourhood of a compact set $A \subseteq \mathbb{R}^d$.
\begin{prop}[Mackay-Tyson] \label{weaktang00}
	Let $X \subset \mathbb{R}^n$ be compact and let $F$ be a compact subset of $X$.  Let $T_k$ be a sequence of bi-Lipschitz maps defined on $\mathbb{R}^n$ with Lipschitz constants $a_k,b_k \geq 1$ such that
\[
a_k \lvert x-y\rvert\  \leq \ \lvert T_k(x) - T_k(y) \rvert \  \leq \  b_k \lvert x-y \rvert \qquad (x,y \in \mathbb{R}^n)
\]
and
\[
\sup_k \, b_k/a_k \ = \ C_0 \ < \  \infty
\]
and suppose that $T_k(F) \cap X \to_{d_\mathcal{H}} \hat{F}$.  Then
\[
\dim_\text{\emph{A}} \hat{F} \  \leq \   \dim_\text{\emph{A}} F.
\]
	The set $\hat{F}$ is called a \emph{very weak tangent} to $F$.
\end{prop}
See \cite[Proposition 6.1.5]{mackaytyson} or \cite[Proposition 2.1]{Mackay11}.  In fact the result given in these references assumes that the $T_k$ are similarities (i.e. $a_k = b_k$ for all $k$) in which case the tangent $\hat{F}$ is called a \emph{weak tangent}.  The very minor modification of their argument required to obtain the version stated here is given in \cite[Proposition 7.7]{Fraser14a}. When applying this proposition to random self-affine carpets it will be more convenient to allow the maps not to be strict similarities.   The hope is that one can construct weak tangents or very weak tangents which are much simpler than the original object and, moreover, have an Assouad dimension that is obvious or at least easy to compute.
\\ \\
One can generalise this result in a number of different directions.  In particular, one can get away with less than convergence to the tangent in the Hausdorff metric.  One generalisation is to construct a \emph{weak pseudo tangent}, which is just a limit in a related \emph{hemimetric}.  This will be useful for us when considering self-similar sets with overlaps, where only a sequence of subsets of the pre-tangents $T_k(F) \cap X$ converges in $d_\mathcal{H}$.  We write $\mathcal{K}(Y)$ for the set of all non-empty compact subsets of a set $Y$.
\begin{prop}\label{pseudoweak}
Let $X \subset \mathbb{R}^d$ be compact and let $F$ be a compact subset of $X$.
Let $T_k$ be a sequence of similarity maps defined on $\mathbb{R}^d$ and suppose that $p_{\mathcal{H}}(\hat F, T_k(F) )\to 0$ as $k\to\infty$ for some non-empty compact set $\hat{F} \in \mathcal{K}(X)$, where
\[
p_{\mathcal{H}}(X,Y) \ = \ \inf \big\{ \varepsilon \geq 0: X \subseteq [Y]_{\varepsilon}\big\} \ = \ \sup_{x\in X}\inf_{y\in Y}|x-y|.
\]
Then $\dim_\text{\emph{A}}\hat F\leq \dim_\text{\emph{A}}F$. We call $\hat F$ the \emph{weak pseudo tangent}.
\end{prop}
A proof of Proposition~\ref{pseudoweak} can be found in \cite{Fraser14}.  Note that we do not need to intersect the pre-tangents $T_k(F)$ with $X$ in the above definition, reflecting the fact that we do not need metric convergence.

\subsection{Proofs concerning random self-similar sets}

\subsubsection{Proof of Theorem \ref{theo2-1}}\label{theo2-1proof}

The proof of Theorem \ref{theo2-1} will closely follow the strategy of Olsen \cite{Olsen11a}, who gave a simple argument demonstrating the sharp upper bound for the Assouad dimension of a deterministic self-similar set satisfying the OSC.  Before beginning the proof we introduce some additional notation.
Let $U$ be the uniform open set given by the UOSC and write $|X|$ for the diameter of a set $X$ and let $u=|U|$.
Fix a realisation $\omega = (\omega_1, \omega_2, \dots) \in\Omega$ and define the set of sequences (or codings) of length $k$ by
\[
\Sigma_{\omega}^{k}=\{(x_{1},x_{2},\hdots,x_{k})\mid x_{j}\in\mathcal{I}_{\omega_j} \text{ for } j=1 \dots, k\},
\]
the set of all finite sequences by
\[
\Sigma_{\omega}^{*}=\bigcup_{k\in\N}\Sigma_{\omega}^{k}
\]
and the set of all infinite sequences by
\[
\Sigma_{\omega}=\{(x_{1},x_{2},\hdots)\mid x_{j}\in\mathcal{I}_{\omega_j}\text{ for } j=1,2,\dots\}.
\]
We write finite words as $\mathbf{x}=(x_{1},\hdots,x_{l})\in\Sigma_{\omega}^{*}$, where $|\mathbf{x}|=l$ is the length of the coding, i.e.\ the number of elements in the sequence.
Given a word $\mathbf{x}=(x_{1},\hdots,x_{l}) \in \Sigma_{\omega}^{*}$, let $\mathbf{x}^{\dagger}=(x_{1},x_{2},\hdots,x_{l-1})$ be the word formed by deleting the final entry. Also, let $S_{\mathbf{x}}=S_{\omega_{1},x_{1}}\circ S_{\omega_{2},x_{2}}\circ\hdots\circ S_{\omega_{l},x_{l}}$ and $\Delta_{\mathbf{x}}=S_{\mathbf{x}}(U)$ and for convenience let $\Delta_{\varepsilon_{0}}=U$, where $\varepsilon_0$ is the empty word.  Observe that
\[
\bigcup_{y\in \mathcal{I}_{\omega_{l+1}}}\Delta_{\mathbf{x}y}\subseteq\Delta_{\mathbf{x}}
\]
with union disjoint, by the UOSC.  We can now rewrite (\ref{attractoreq}) as
\[
F_{\omega} \, = \, \bigcap_{k}\bigcup_{\mathbf{x}\in\Sigma_{\omega}^{k}} \overline{\Delta_{\mathbf{x}}}.
\]
We have $|\Delta_{\mathbf{x}}|=c_{\omega_{1},x_{1}}c_{\omega_{2},x_{2}}\hdots c_{\omega_{l},x_{l}}u$ and write $c_{\mathbf{x}}=|\Delta_{\mathbf{x}}|$ for brevity.  For $r \in (0,1]$, let $\Gamma_{\omega}(r)$ be the set of codings $\mathbf{x}$ for which the associated $\Delta_{\mathbf{x}}$ has diameter approximately $r$, that is
\[
\Gamma_{\omega}(r)=\{\mathbf{x}\in\Sigma_{\omega}^{*}\mid c_{\mathbf{x}}<r\leq c_{\mathbf{x}^{\dagger}}\}.
\]
The following lemma shows that the number of sets $\overline{\Delta_{\mathbf{x}}}$ with diameter approximately $r$ that intersect a closed ball $B(z,r)$ centred at $z\in F_{\omega}$ of radius $r$ is bounded by a constant not depending on $r$ and $z$. This is a simple generalisation of results in \cite{Olsen11a, Hutchinson81} to the random setting and is included here for completeness.
\begin{lma} \label{upperasslem1}
Under the same assumptions as Theorem \ref{theo2-1}, we have
\[
|\{\mathbf{x}\in\Gamma_{\omega}(r)\mid \overline{\Delta_{\mathbf{x}}}\cap B(z,r)\neq\varnothing\}| \, \leq \,  (4/c_{\text{min}})^{d}
\]
for all $z\in F_{\omega}$ and  $r \in (0,1]$,  where $c_{\text{min}}=\min_{i\in\Lambda, j\in\mathcal{I}_{i}}c_{i,j}$.
\end{lma}

\begin{proof}
Fix $z\in F_{\omega}$ and $r>0$.  Let $\Xi=\{\mathbf{x}\in\Gamma_{\omega}(r)\mid \overline{\Delta_{\mathbf{x}}} \cap B(z,r)\neq\varnothing\}$ and suppose the ambient space is $\R^{d}$. We have
\[
|\Xi|(rc_{\text{min}})^{d} \ = \ \sum_{\mathbf{x}\in\Xi}(rc_{\text{min}})^{d} \ \leq \ \sum_{\mathbf{x}\in\Xi}c_{\mathbf{x}}^{d} \ = \ \sum_{\mathbf{x}\in\Xi}|\Delta_{\mathbf{x}}|^{d}.
\]
But since $\overline{\Delta_{\mathbf{x}}}\cap B(z,r)\neq\varnothing$ and $|\Delta_{\mathbf{x}}|<r$ we find that $\Delta_{\mathbf{x}}\subseteq B(z,2r)$ for all $\mathbf{x} \in \Xi$ and since the sets $\Delta_{\mathbf{x}}$ are pairwise disjoint we have
\[
|\Xi|(rc_{\text{min}})^{d} \ \leq \ \sum_{\mathbf{x}\in\Xi}|\Delta_{\mathbf{x}}|^{d} \ \leq \  \mathcal{L}^{d}(B(z,2r)) \ \leq \ (4r)^{d},
\]
where $\mathcal{L}^{d}$ is the $d$-dimensional Lebesgue measure.  It follows that $|\Xi|\leq(4/c_{\text{min}})^{d}$ as required.
\end{proof}

Write $s=\max_{i\in \Lambda}(\dim_{\text{A}}F_{\underline i})$.

\begin{lma} \label{upperasslem2}
Under the same assumptions as Theorem \ref{theo2-1}, we have
\[
|\Gamma_{\omega}(r)| \, \leq \, u^s \, c_{\min}^{-s} \, r^{-s}
\]
for all $r\in (0,1]$.
\end{lma}

\begin{proof}
Fix $r\in (0,1]$ and observe that since the Assouad dimension of each deterministic attractor is given by the appropriate version of the Hutchinson-Moran formula, we have
\[
\sum_{j \in \mathcal{I}_i} c_{i,j}^s \leq 1
\]
for all $i \in \Lambda$.  By repeated application of this, it follows that
\[
u^s\  \geq \  \sum_{\mathbf{x}\in\Gamma_{\omega}(r)} c_{\mathbf{x}}^s \  \geq \  \sum_{\mathbf{x}\in\Gamma_{\omega}(r)} (c_{\min} r)^s \ =  \ |\Gamma_{\omega}(r)| \, c_{\min}^s r^s ,
\]
which proves the lemma.
\end{proof}
We can now prove Theorem \ref{theo2-1}.  Fix $z \in F_\omega$, $R \in (0,1]$ and $r \in (0,R]$.  Clearly
\[
B(z,R)\cap F_\omega \ \subseteq \ \bigcup_{\substack{\mathbf{x}\in\Gamma_{\omega}(R)\\ \overline{\Delta_{\mathbf{x}}}\cap B(z,R)\neq\varnothing}} \overline{\Delta_{\mathbf{x}}}
\]
and for each such set $\overline{\Delta_{\mathbf{x}}}$ in the above decomposition we have
\[
\overline{\Delta_{\mathbf{x}}} \ \subseteq \ \bigcup_{\mathbf{y}\in\Gamma_{\sigma(\mathbf{x},\omega)}(r/R)}\overline{\Delta_{\mathbf{x}\mathbf{y}}},
\]
where for clarity we have written $\sigma(\mathbf{x},\omega)=\sigma^{|\mathbf{x}|}(\omega)$, with $\sigma$ the usual shift map. These observations combine to give
\[
B(z,R)\cap F_\omega \ \subseteq \  \bigcup_{\substack{\mathbf{x}\in\Gamma_{\omega}(R)\\ \overline{\Delta_{\mathbf{x}}}\cap B(z,R)\neq\varnothing}}\; \bigcup_{\mathbf{y}\in\Gamma_{\sigma(\mathbf{x},\omega)}(r/R)}\overline{\Delta_{\mathbf{x}\mathbf{y}}},
\]
which is an $r$-cover of $B(z,R)\cap F_\omega$, yielding
\begin{eqnarray*}
N_{r}(B(z,R)\cap F_{\omega}) &\leq&  \sum_{\substack{\mathbf{x}\in\Gamma_{\omega}(R)\\\overline{\Delta_{\mathbf{x}}}\cap B(z,R)\neq\varnothing}}\; \rvert\Gamma_{\sigma(\mathbf{x},\omega)}(r/R) \lvert \\ \\
&\leq&    \sum_{\substack{\mathbf{x}\in\Gamma_{\omega}(R)\\ \overline{\Delta_{\mathbf{x}}}\cap B(z,R)\neq\varnothing}} u^s \, c_{\min}^{-s} \,\left(\frac{R}{r}\right)^s  \qquad \text{by Lemma \ref{upperasslem2} since $\omega \in \Omega$ was arbitrary} \\ \\
&\leq&   (4/c_{\text{min}})^{d} \,  u^s \, c_{\min}^{-s} \,\left(\frac{R}{r}\right)^s \qquad \text{by Lemma \ref{upperasslem1}},
\end{eqnarray*}
which proves the theorem.

\subsubsection{Proof of Theorem \ref{theo2-2}}\label{theo2-2proof}

In order to prove the almost sure lower bound we identify a `good set' of full measure within which we can prove the lower bound surely.   Fix $i\in \Lambda$ which maximises $\dim_{\text{\emph{A}}}F_{\underline i}$.  The \emph{good set} $G_{i}$ is the set of all realisations $\omega\in\Omega$ such that there are arbitrarily long subwords consisting only of the letter $i$.
Equivalently, let
\[
G_{i}=\{\omega=(\omega_{1},\omega_{2},\hdots)\in\Omega\mid\forall n\in\N, \exists k\in\N, \exists m\geq n \text{ such that }\omega_{j}=i\text{  for all  }k\leq j<k+m\}.
\]
The following Lemma follows from a standard Borel-Cantelli argument, that we will not give in this paper.
\begin{lma}\label{exceptionalLemma}
Let $\mu$ be the Bernoulli probability measure on $\Omega$  defined  in~\ref{def_m}.  Then for all $i \in \Lambda$ almost all $\omega\in\Omega$ are contained in $G_{i}$, that is $\mu(G_{i})=1$.
\end{lma}

Let $i \in \Lambda$ and $\omega = (\omega_1, \omega_2, \dots) \in G_i$.  We will show that the deterministic attractor $F_{\underline i}$ is a weak pseudo tangent to $F_{\omega}$ and this is sufficient to prove Theorem \ref{theo2-2} in light of Proposition \ref{pseudoweak}.  Without loss of generality we shall assume that the ambient space $X = [0,1]^{d}$, for some $d \in \N$.  Note that, since we do not assume any separation conditions, the existence of complicated overlaps mean that $F_{\underline i}$ may not be a weak (or very weak) tangent to $F_{\omega}$.
\\ \\
Since $\omega \in G_i$, for every $N$ we can find $k_N$ such that $\omega_j = i$ for all $k_N +1 \leq j \leq k_N+N$.  Choose any sequence $(i_1, i_2, \dots, i_{k_N})$ with $i_{j'} \in \mathcal{I}_{\omega_{j'}}$ for all $1 \leq j' \leq k_N$ and let $T_{k_N}$ be the similarity given by
\[
T_{k_N} = \big( S_{\omega_1, i_1} \circ \dots \circ S_{\omega_{k_N},i_{k_N}} \big)^{-1}.
\]
Write $c^i_{\max} = \max_{j \in \mathcal{I}_i} c_{i,j} \in (0,1)$.  It follows that
\[
F_{\sigma^{k_N}(\omega)} \, \subseteq \,  T_{k_N} (F_\omega)
\]
and therefore, since the first $N$ symbols in $\sigma^{k_N}(\omega)$ are all $i$,
\[
F_{\underline i} \, \subseteq \,   [T_{k_N}F_{\omega}]_{(c^i_{\max})^N}.
\]
This proves that
\[
p_{\mathcal{H}} \big(F_{\underline i}, \, T_{k_N}F_{\omega} \big) \  \leq  \ (c^i_{\max})^N \  \to \  0
\]
as $N \to \infty$. Thus $F_{\underline i}$ is a weak pseudo tangent to $F_{\omega}$, choosing the sequence of maps $\{T_{k_N}\}_{N \in \N}$.

\subsubsection{Proof of Theorem \ref{exceptionalResult}}\label{exceptionalProof2}

Using the mass distribution principle \cite[Chapter 4]{FractalGeo}, it is easy to see that $\dim_{\text{H}}\Omega=\log N/\log 2$ where $N$ is the cardinality of $\Lambda$ and the `2' comes from our choice of metric $d(\omega,\nu)=2^{-\omega\wedge\nu}$. Let $\beta = \max_{i\in \Lambda} \, \dim_{\text{A}}F_{\underline i}$ and
\[
\Lambda_E = \{ i \in \Lambda \mid \dim_{\text{A}}F_{\underline i} < \beta \}
\]
which by assumption is non-empty and by definition is strictly smaller than $\Lambda$. Let
\[
E_{n}=\big\{\omega\in\Omega\mid   \text{ if, for some $k\in\N$,  $\omega_{j} \notin \Lambda_E$ for all $j = k, k+1, \dots, k+n-1$, then } \omega_{k+n} \in \Lambda_E \big\},
\]
i.e. the set of all sequences such that the length of subwords consisting only of letters which maximise the Assouad dimension is bounded above by $n$.  First we will show that for $\omega \in E_{n}$, we have $\dim_{\text{A}} F_\omega < \beta$. Let $\Lambda^\dagger_{n}$ be a new alphabet consisting of all combinations of words of length at most $n$ (including length zero) over $\Lambda \setminus \Lambda_E$ concatenated with an element of $\Lambda_E$, that is
\[
\Lambda^{\dagger}_{n}=\{vw\mid v\in \cup_{k=0}^n (\Lambda\setminus\Lambda_{E})^{k}\text{ and }w\in\Lambda_{E}\}.
\]
 To each word (now identified as a letter) in $\Lambda^\dagger_{n}$ we associate the IFS formed by combining the IFSs corresponding to $\Lambda$ in the natural way. Since the UOSC was satisfied, it is easy to see that the similarity dimension of each such deterministic IFS is strictly less than $\beta$ since they are all influenced by an IFS associated to an element of $\Lambda_E$.  Moreover, every word in $E_n$ can be obtained as a word over $\Lambda^\dagger_{n}$ and so the required dimension drop follows from Theorem \ref{theo2-2}.  It follows from this that the exceptional set from Theorem \ref{exceptionalResult} contains
\[
E \ := \ \bigcup_{n=1}^{\infty} E_n
\]
and so it suffices to prove that the Hausdorff dimension of $E$ is $\log N/\log 2$.  Now consider the finite set $\Lambda'$ consisting of all possible words of length $\lceil n/2 \rceil$.  We could have equivalently defined $\Omega$ in terms of those words rather than the individual symbols $\Lambda$ where, abusing notation slightly,  $\Omega=\Lambda^{\N}=\Lambda'^{\N}$.
Consider $\Lambda'$ and remove the words consisting only of letters from $\Lambda \setminus \Lambda_E$ forming a new set $\Lambda''$. If one considers $E'_{n}=\Lambda''^{\N}$ one notes that several combinations are now no longer possible.
Crucially it restricts the length of subwords over $\Lambda \setminus \Lambda_E$ to $2(\lceil n/2\rceil-1)$, which corresponds to two concatenated elements of $\Lambda''$, one starting with symbol $j \in \Lambda_E$ followed by letters from $\Lambda \setminus\Lambda_E$ and the second word starting with letters from $\Lambda \setminus\Lambda_E$ but ending with $j \in \Lambda_E$.
Since $2(\lceil n/2\rceil-1)\leq n$ we have that elements of $E'_{n}$ have more restrictive conditions than $E_{n}$ and so $E'_{n}\subseteq E_{n}$.  Let $\nu$ be the uniform Bernoulli measure on $E'_n$, given by a uniform probability vector associated with $\Lambda''$, and let
\[
\alpha_{n}=\frac{\log (N^{\lceil n/2\rceil}-\lvert \Lambda \setminus\Lambda_E \rvert^{\lceil n/2\rceil})}{\lceil n/2\rceil\log 2}.
\]
Let $U_k \subseteq E'_n$ be a cylinder of length $k$ (over $\Lambda''$) and observe that $\nu(U_{k})=(N^{\lceil n/2\rceil}-\lvert \Lambda \setminus\Lambda_E \rvert^{\lceil n/2\rceil})^{-k}$ and $|U_{k}|=2^{-k\lceil n/2\rceil}$ and so
\begin{eqnarray*}
|U_{k}|^{\alpha_{n}}&=& 2^{-k\lceil n/2\rceil \log (N^{\lceil n/2\rceil}-\lvert \Lambda \setminus\Lambda_E \rvert^{\lceil n/2\rceil})/(\lceil n/2\rceil\log 2)} \\
&=& 2^{-k \log (N^{\lceil n/2\rceil}-\lvert \Lambda \setminus\Lambda_E \rvert^{\lceil n/2\rceil})/\log 2} \\
&=&(N^{\lceil n/2\rceil}-\lvert \Lambda \setminus\Lambda_E \rvert^{\lceil n/2\rceil})^{-k} \\
&=& \nu(U_{k})
\end{eqnarray*}
and thus by the mass distribution principle $\dim_{\text{H}}E'_{n}\geq\alpha_{n}$.  Finally $\dim_{\text{H}}E \, = \, \sup_n \dim_{\text{H}}E_n \,  \geq \,  \sup_n \alpha_n \,  = \, \log N/\log 2$,  as required.

\subsection{Proofs concerning random self-affine carpets}

\subsubsection{Preliminary results and random approximate $R$-squares}

In this section we introduce random approximate $R$-squares, which will be heavily relied on in both the upper bound and the lower bound.  Fix  $\omega = (\omega_1, \omega_2, \dots) \in \Omega$ and $R \in (0,1)$ and let $k_1^\omega(R)$ and $k_2^\omega(R)$ be the unique natural numbers satisfying
\begin{equation} \label{k1def}
\prod_{i=1}^{k_1^\omega(R)} n_{\omega_i}^{-1} \, \leq \,  R \, <  \, \prod_{i=1}^{k_1^\omega(R)-1} n_{\omega_i}^{-1}
\end{equation}
and
\begin{equation} \label{k2def}
\prod_{i=1}^{k_2^\omega(R)} m_{\omega_i}^{-1} \, \leq \,  R \, <  \, \prod_{i=1}^{k_2^\omega(R)-1} m_{\omega_i}^{-1}
\end{equation}
respectively. Also let
\[
m_{\max} = \max_{i \in \Lambda} \,  m_i \qquad \text{and} \qquad n_{\max} = \max_{i \in \Lambda} \,  n_i.
\]
A rectangle $[a,b] \times [c,d] \subseteq [0,1]^2$ is called a random approximate $R$-square if it is of the form
\[
S \big( [0,1]^2 \big) \cap \Big( \pi_1 \big( T\big( [0,1]^2 \big) \big) \times [0,1] \Big),
\]
where $\pi_1 : (x,y) \mapsto x$ is projection onto the first coordinate and
\[
S \ = \ S_{\omega_1, i_1} \circ \cdots \circ S_{\omega_{k_1^\omega(R)}, i_{k_1^\omega(R)}}
\]
and
\[
T \ = \ S_{\omega_1, i_1} \circ \cdots \circ S_{\omega_{k_2^\omega(R)}, i_{k_2^\omega(R)}}
\]
for some common sequence $i_1, i_2, \dots $ with $i_j \in \mathcal{I}_{\omega_j}$ for all $j$.  The use of the term `random' indicates that the family of approximate $R$-squares depends on the random sequence $\omega$ and observe that such rectangles are indeed approximately squares of side length $R$ because the base
\[
b-a  \ =  \ \prod_{i=1}^{k_2^\omega(R)} m_{\omega_i}^{-1} \ \in \ (m_{\max}^{-1}R , R] \qquad \qquad \text{by (\ref{k2def})}
\]
and the height
\[
d-c  \ =  \ \prod_{i=1}^{k_1^\omega(R)} n_{\omega_i}^{-1} \ \in \ (n_{\max}^{-1}R , R] \qquad \qquad \text{by (\ref{k1def})}.
\]
Approximate squares are a standard tool in the study of self-affine carpets.

\subsubsection{Proof of Theorem \ref{upperboundcarpets}} \label{upperboundcarpetsproof}

Fix $\omega = (\omega_1, \omega_2, \dots) \in \Omega$,  $R \in (0,1)$ and $r \in (0,R)$.  For $k,l \in \mathbb{N}$ and $i \in \Lambda$ let
\[
\mathcal{N}_i(k,l) \ = \ \# \big\{ j = k, k+1, \hdots, l \  :  \  \omega_j = i  \big\}.
\]
We wish to bound $N_r\big( B(x,R) \cap F_\omega\big)$ up to a constant uniformly over $x \in F_\omega$, but since there exists a constant $K \geq 1$ depending on $m_{\max}$ and $n_{\max}$ such that for any $x \in F_\omega$, $B(x,R)$  is contained in fewer than $K$ random approximate $R$-squares, it suffices to bound $N_r (Q \cap F_\omega)$ up to a constant uniformly over all random approximate $R$-squares, $Q$.  We will adopt the version of $N_r(\cdot)$ which uses covers by squares of sidelength $r$.  Fix such a $Q$ and observe that $Q \cap F_\omega$ can be decomposed as the union of the parts of $F_\omega$ contained inside rectangles of the form
\[
X \ = \ S_{\omega_1, i_1} \circ \cdots \circ S_{\omega_{k_2^\omega(R)}, i_{k_2^\omega(R)}} \big([0,1]^2\big)
\]
for some $i_1, i_2, \dots $ with $i_j \in \mathcal{I}_{\omega_j}$ for all $j$. Moreover, the number of such rectangles in this decomposition can be bounded above by
\[
\prod_{i \in \Lambda} B_i^{\mathcal{N}_i (k_1^\omega(R)+1, k_2^\omega(R)  )}.
\]
Now, let us continue to iterate the construction of $F_\omega$ inside such a rectangle $X$, i.e. by breaking it up into smaller basic rectangles. Assuming $k_1^\omega(r) > k_2^\omega(R)$ continue iterating until level $k_1^\omega(r)$ where each $X \cap F_\omega$ can be written as the union of parts of $F_\omega$ inside rectangles of the form
\[
Y \ = \ S_{\omega_1, i_1} \circ \cdots \circ S_{\omega_{k_1^\omega(r)}, i_{k_1^\omega(r)}} \big([0,1]^2\big)
\]
for some $i_1, i_2, \dots $ with $i_j \in \mathcal{I}_{\omega_j}$ for all $j$.  Note that this time we use words of length $k_1^\omega(r)$.  Writing $N_i = \lvert \mathcal{I}_i \rvert $ ($i \in \Lambda$), we can bound the number of rectangles of the form $Y$ used to decompose a rectangle of the form $X$ by
\[
\prod_{i \in \Lambda} N_i^{\mathcal{N}_i (k_2^\omega(R)+1, k_1^\omega(r)  )}.
\]
If $k_1^\omega(r) \leq  k_2^\omega(R)$, then we leave $X$ alone and set $Y=X$, corresponding to $\mathcal{N}_i (k_2^\omega(R)+1, k_1^\omega(r)  ) = 0$ for each $i$.  Note that each rectangle $Y$ in the new decomposition is a rectangle with height
\[
\prod_{i=1}^{k_1^\omega(r)} n_{\omega_i}^{-1}  \ \leq  \ r,
\]
and we are trying to cover it by squares of side length $r$.  Thus to give an efficient estimate on $N_r(Y)$ we need only worry about covering $\pi_1(Y)$ and we can certainly do this using no more than
\[
\prod_{i \in \Lambda} A_i^{\mathcal{N}_i (k_1^\omega(r)+1, k_2^\omega(r)  )}
\]
such squares.  Combining the above estimates and using the fact that for all $i \in \Lambda$, $N_i \leq A_iB_i$  yields
\begin{eqnarray*}
N_r (Q \cap F_\omega) &\leq &  \bigg( \prod_{i \in \Lambda} B_i^{\mathcal{N}_i (k_1^\omega(R)+1, k_2^\omega(R)  )} \bigg) \bigg( \prod_{i \in \Lambda} N_i^{\mathcal{N}_i (k_2^\omega(R)+1, k_1^\omega(r)  )} \bigg)\bigg( \prod_{i \in \Lambda} A_i^{\mathcal{N}_i (k_1^\omega(r)+1, k_2^\omega(r)  )} \bigg) \\ \\
& \leq & \prod_{i \in \Lambda} A_i^{\mathcal{N}_i (k_2^\omega(R)+1, k_2^\omega(r)  )} B_i^{\mathcal{N}_i (k_1^\omega(R)+1, k_1^\omega(r)  )}.
\end{eqnarray*}
Now that this estimate has been established, the desired upper bound follows by careful algebraic manipulation.  In particular,
\begin{eqnarray*}
N_r (Q \cap F_\omega) &\leq &  \bigg( \prod_{i \in \Lambda} A_i^{\mathcal{N}_i (k_2^\omega(R)+1, k_2^\omega(r)  )}  \bigg) \bigg( \prod_{i \in \Lambda}  B_i^{\mathcal{N}_i (k_1^\omega(R)+1, k_1^\omega(r)  )} \bigg) \\ \\
&= &  \prod_{i \in \Lambda} \bigg( m_i^{\mathcal{N}_i (k_2^\omega(R)+1, k_2^\omega(r)  )}  \bigg)^{\log A_i / \log m_i} \prod_{i \in \Lambda}  \bigg( n_i^{\mathcal{N}_i (k_1^\omega(R)+1, k_1^\omega(r)  )} \bigg)^{\log B_i / \log n_i} \\ \\
&\leq &  \bigg( \prod_{i \in \Lambda} m_i^{\mathcal{N}_i (k_2^\omega(R)+1, k_2^\omega(r)  )}  \bigg)^{\max_{i \in \Lambda} \log A_i / \log m_i} \bigg( \prod_{i \in \Lambda}   n_i^{\mathcal{N}_i (k_1^\omega(R)+1, k_1^\omega(r)  )} \bigg)^{ \max_{i \in \Lambda}  \log B_i / \log n_i} \\ \\
&=&  \Bigg( \frac{\prod_{i \in \Lambda} m_i^{-\mathcal{N}_i (1, k_2^\omega(R) )}}{\prod_{i \in \Lambda} m_i^{-\mathcal{N}_i (1, k_2^\omega(r)  )}}  \Bigg)^{\max_{i \in \Lambda} \log A_i / \log m_i} \Bigg( \frac{\prod_{i \in \Lambda}   n_i^{-\mathcal{N}_i (1, k_1^\omega(R) )}}{\prod_{i \in \Lambda}   n_i^{-\mathcal{N}_i (1, k_1^\omega(r)  )}} \Bigg)^{ \max_{i \in \Lambda}  \log B_i / \log n_i} \\ \\
&=&  \Bigg( \frac{\prod_{i =1}^{k_2^\omega(R)} m_{\omega_i}^{-1}}{\prod_{i =1}^{k_2^\omega(r)} m_{\omega_i}^{-1}} \Bigg)^{\max_{i \in \Lambda} \log A_i / \log m_i} \Bigg( \frac{\prod_{i =1}^{k_1^\omega(R)} n_{\omega_i}^{-1}}{\prod_{i =1}^{k_1^\omega(r)} n_{\omega_i}^{-1}}\Bigg)^{ \max_{i \in \Lambda}  \log B_i / \log n_i} \\ \\
&\leq&  \Bigg( \frac{R}{m_{\max}^{-1} \,  {r}} \Bigg)^{\max_{i \in \Lambda} \log A_i / \log m_i} \Bigg( \frac{R}{n_{\max}^{-1} \,  {r}}\Bigg)^{ \max_{i \in \Lambda}  \log B_i / \log n_i} \qquad \qquad \text{by (\ref{k1def}) and (\ref{k2def})} \\ \\
&\leq&  m_{\max} \, n_{\max} \, \bigg( \frac{R}{r} \bigg)^{\max_{i \in \Lambda} \log A_i / \log m_i \ + \ \max_{i \in \Lambda}  \log B_i / \log n_i},
\end{eqnarray*}
which proves that
\[
\dim_\text{A} F_\omega \ \leq  \ \max_{i \in \Lambda}  \, \frac{\log A_i}{\log m_i} \ + \ \max_{i \in \Lambda} \,  \frac{\log B_i}{ \log n_i}
\]
and since $\omega \in \Omega$ was arbitrary this proves the desired result.

\subsubsection{Proof of Theorem \ref{lowerboundcarpets}} \label{lowerboundcarpetsproof}

In light of Theorem \ref{upperboundcarpets}, all that remains is to prove the almost sure lower bound.  In order to do this we identify a `good set' of full measure within which we can prove the lower bound surely, similar to Theorem \ref{theo2-2}.  First fix $i\in\Lambda$ which maximises $\log A_i / \log m_i$ and $j\in\Lambda$ which maximises $\log B_j/\log n_j$.  Of course $i$ and $j$ may be different, and this is the more interesting case which leads to examples such as those in Section \ref{carpetscounterex}.
\\ \\
A first guess for the good set might be the set of strings containing arbitrarily long runs of $j$ followed by the same number of $i$.  This is philosophically the correct approach, but does not work because the point where the string is required to change from $j$ to $i$ depends crucially on the stage one is at in the sequence.  Since one may have to wait much longer than $O(n)$ steps to get a string of $n$ $j$-s followed by $n$ $i$-s, by the time it occurs the eccentricity of the rectangles in the construction will be so large that switching from $j$ to $i$ after $n$ steps in the approximate square is not enough to obtain the desired tangent.  A second approach might be to look for strings of $j$-s followed by $i$-s where the number of $j$-s depends on the starting point of the string (in fact the dependence would be linear), however, this approach also fails because one cannot guarantee that such strings exist infinitely often almost surely.  Our solution is to recognise that one needs a long string of $i$-s and a long string of $j$-s to get the necessary tangent, but these strings do not have to be next to each other.
\\ \\
The good set $G_{i,j} \subseteq \Omega$ is defined to be
\begin{eqnarray*}
G_{i,j} = \Big\{ \omega = (\omega_1, \omega_2, \dots) \in \Omega & \mid&  \text{there exists a sequence of pairs } (R_l,n_l) \in (0,1) \times \mathbb{N} \\
&\,& \text{with $R_l \to 0$ and $n_l \to \infty$} \text{ with } n_l \leq k_2^\omega(R_l) - k_1^\omega(R_l) \text { such that } \\
&\,& \omega_{i'} = j \text{ for all } i'=k_1^\omega(R_l)+1, \dots, k_1^\omega(R_l)+n_l \text { and } \\
&\,& \omega_{i'} = i \text{ for all } i'=k_2^\omega(R_l)+1, \dots, k_2^\omega(R_l)+n_l \Big\}.
\end{eqnarray*}

\begin{lma} \label{affinefullmeasure}
The good set has full measure in $\Omega$, i.e. $\mu(G_{i,j}) = 1$.
\end{lma}

\begin{proof}
For $n \in \mathbb{N}$ let
\[
l(n)=\left\lceil\frac{-\log2}{\log(1-p_{j}^{n}p_{i}^{n})}\right\rceil
\]
and let
\[
\theta = \frac{\max_{i \in \Lambda} \log n_i}{\min_{i \in \Lambda} \log m_i} > 1.
\]
Also, for $n \in \mathbb{N}$ and $m =1, \dots, l(n)+1$, we define numbers $K(n), K_n(m) \in \mathbb{N}$ inductively by
\[
K(1) = 1,
\]
\[
K_n(1) = K(n)
\]
\[
K_n(m+1) =  \theta K_n(m) + n \qquad (m =1, \dots, l(n)),
\]
\[
K(n+1) = K_n(l(n)+1).
\]
These numbers are arranged as follows and will form partitions of the natural numbers:
\[
 \cdots < K(n) = K_n(1) < K_n(2) < \cdots < K_n(l(n)+1) = K(n+1) < \cdots.
\]
For $\omega \in \Omega$, $n \in \mathbb{N}$ and $m\in\{1, \dots, l(n)\}$, let $K_n^\omega(m) = k_2^\omega(R)$ for 
\[
R = \prod_{i=1}^{K_n(m)} n_{\omega_i}^{-1}
\]
and let
\begin{eqnarray*}
E_{n}(m) = \Big\{ \omega = (\omega_1, \omega_2, \dots) \in \Omega & \mid&  \omega_{i'} = j \text{ for all } i'= K_n(m)+1, \dots, K_n(m)+n \text { and } \\
&\,& \omega_{i'} = i \text{ for all } i'=K_n^\omega(m)+1, \dots, K_n^\omega(m)+n \Big\}
\end{eqnarray*}
observing that $n \ll  K_n^\omega(m) - K_n(m)$ for large $n$. Finally, let
\[
E_n = \bigcup_{m=1}^{l(n)} E_n(m).
\]
It follows from these definitions that
\[
\bigcap_{k \in \mathbb{N}} \ \bigcup_{n > k} \ E_{n} \subseteq G_{i,j}
\]
and, moreover, the `events' $\{E_n\}_{n \in \mathbb{N}}$ are independent because they concern properties of $\omega$ at disjoint parts of the sequence.  This can be seen since $K_n^\omega(m)+n \leq \theta K_n(m)+n = K_n(m+1)$.  Also, for a fixed $n$, the events $\{E_n(m)\}_{m=1}^{l(n)}$ are independent. We have
\begin{eqnarray*}
\mu(E_{n})&=&1- \prod_{m=1}^{l(n)} \mu \big( \Omega \setminus E_n(m) \big) \\
&=&1-\big(1-p_{j}^{n}p_{i}^{n}\big)^{l(n)} \\
&\geq&  1-(1-p_{j}^{n}p_{i}^{n})^{-\log2/\log(1-p_{j}^{n}p_{i}^{n})} \\
& =& 1/2.
\end{eqnarray*}
Therefore
\[
\sum_{n\in\N} \mu( E_{n}) \ \geq \ \sum_{n\in\N} \, 1/2 \ = \ \infty
\]
and since the events $E_{n}$ are independent the Borel-Cantelli Lemma implies that
\[
\mu( G_{i,j})  \ \geq \ \mu \left( \bigcap_{k \in \mathbb{N}} \ \bigcup_{n> k} \ E_{n} \right) = 1
\]
as required.
\end{proof}

We can now prove Theorem \ref{lowerboundcarpets}.  Fix $\omega \in G_{i,j}$ and consider a column of the defining pattern for $\mathbb{I}_j$ containing a maximal number of chosen rectangles $B_j$.  If there is more than one such column, then choose one arbitrarily.  This column induces a natural IFS of similarities on the unit interval, consisting of $B_j$ maps with contraction ratios $n_j^{-1}$ and satisfying the OSC.  Let $E_j$ denote the self-similar attractor of this IFS and for $l \in \mathbb{N}$, let $E_j^l$ denote the $l$th level of the construction, i.e., the union of $(B_j)^l$ intervals of length $n_j^{-l}$ corresponding to images of $[0,1]$ under compositions of $l$ maps from the induced column IFS.  Also, consider the IFS $\mathbb{I}_i$ and let $\pi_1(F_{\underline{i}})$ denote the projection onto the first coordinate of the attractor of $\mathbb{I}_i$, which is also a self-similar set satisfying the OSC.  We will now show that $ \pi_1(F_{\underline{i}}) \times E_j$ is a very weak tangent to $F_\omega$.
\\ \\
For a random approximate square $Q$, let $T^Q$ be the uniquely defined affine map given by the composition of a non-negative diagonal matrix and a translation which maps $Q$ to $[0,1] \times [0,1]$.  Let $(R_l, n_l) \in (0,1) \times \mathbb{N}$ be a pair which together with $\omega$ satisfy the definition of $G_{i,j}$ and consider the family of random approximate $R_l$ squares.  Since $\omega_{i'} = j \text{ for all } i'=k_1^\omega(R_l)+1, \dots, k_1^\omega(R_l)+n_l$, by keeping track of the maximising column mentioned above we can choose $Q$ satisfying
\[
T^Q(Q) \ \subseteq \pi_1\Big(F_{\sigma^{k_2^\omega(R)}(\omega)} \Big) \times E_j^{n_l}.
\]
Moreover, by decomposing $E_j^{n_l}$ into its basic intervals of length $n_j^{-l}$, we see that within each corresponding rectangle in $T^Q(Q)$ (which has height $n_j^{-l}$), one finds affinely scaled copies of $F_{\sigma^{k_2^\omega(R)}(\omega)}$.  Since $\omega_{i'} = i \text{ for all } i'=k_2^\omega(R_l)+1, \dots, k_2^\omega(R_l)+n_l $, this implies that $T^Q(Q)$ occupies every basic rectangle at the $n^l$th stage of the construction of $\pi_1(F_{\underline{i}}) \times E_j$.  Since such rectangles have base $m_i^{- n_l}$ and height $n_j^{- n_l}$ this yields
\[
d_\mathcal{H} \Big( T^Q(Q), \, \pi_1(F_{\underline{i}}) \times E_j \Big) \ \leq \  \Big(  m_i^{-2 n_l} +n_j^{-2 n_l}  \Big)^{1/2}.
\]
This is sufficient to show that $\pi_1(F_{\underline{i}}) \times E_j$ is a very weak tangent to $F_\omega$ because we can choose our sequence of maps to be $T^{Q}$ for a sequence of random approximate squares $Q$ satisfying the above inequality, but with $n_l \to \infty$, giving the desired convergence. Moreover, for any random approximate $R$-square $Q$ we have
\[
R^{-1} \lvert x-y\rvert\  \leq \ \lvert T^Q(x) - T^Q(y) \rvert \  \leq \  n_{\max} R^{-1} \lvert x-y \rvert \qquad (x,y \in \mathbb{R}^2),
\]
and so the maps satisfy the conditions required in Proposition \ref{weaktang00}. It follows that
\begin{eqnarray*}
\dim_\text{A} F_\omega & \geq &  \dim_\text{A} \big( \pi_1(F_{\underline{i}}) \times E_j \big) \qquad \text{by Proposition \ref{weaktang00}}\\ 
 & \geq &  \dim_\text{H} \big( \pi_1(F_{\underline{i}}) \times E_j \big)\\ 
& \geq &  \dim_\text{H} \pi_1(F_{\underline{i}})  \ + \ \dim_\text{H} E_j \qquad \text{by \cite[Corollary 7.4]{FractalGeo}}\\ 
& = & \frac{\log A_i}{\log m_i}   \ + \ \frac{\log B_j}{\log n_j}
\end{eqnarray*}
as required.

\subsection{Proof of Theorem \ref{baireassouad}} \label{baireassouadproof}

 Let $s=\sup_{u \in \Omega} \, \dim_\text{{A}} F_u $.  We will show that the set
 \[
 A \ = \ \{ \omega \in \Omega \ : \  \dim_\text{{A}} F_\omega \geq s\}
 \]
 is residual, from which Theorem \ref{baireassouad} follows.
 \\ \\
   First we recall some useful functions. Let $\Psi: \big(\Omega, d \big) \to \big(\mathcal{K}(X), d_\mathcal{H}\big)$ be defined by $\Psi(\omega) = F_\omega$ and observe that it is continuous.
   For $x \in X$ and $R \in (0,1]$ let  $\beta^0_{x,R}: \mathcal{K}(X)  \to \mathcal{P}(X)$ by
 \[
 \beta^0_{x,R}(F) = B^0(x,R) \cap F,
 \]
where $B^0(x,R)$ is the open ball centered at $x$ with radius $R$, and $\mathcal{P}(X)$ is the power set of $X$ (the images need not be compact). Also, for $r \in (0,1]$, let $M_r(F)$ denote the maximum number of \emph{closed} sets in an $r$-packing of $F\subseteq X$, where an $r$-packing of $F$ is a pairwise disjoint collection of closed balls centered in $F$ of radius $r$.
 It was shown in \cite[Lemma 5.2]{Fraser14a} that the map $M_r \circ \beta^0_{x,R} : \mathcal{K}(X) \to \mathbb{R}$ is lower semicontinuous.  It thus  follows from the continuity of $\Psi$, that the function $\Xi : = M_r \circ \beta^0_{x,R} \circ \Psi : \Omega \to \mathbb{R}$ is lower semicontinuous. We have
 \begin{eqnarray*}
 	A 	&=& \Bigg\{\omega \in \Omega \ : \  \text{for all } n \in \mathbb{N},  C, \rho>0, \text{ there exists } x \in X  \text{ and } 0<r<R< \rho, \text{ such that } \\ \\
 	&\,& \qquad \qquad \qquad \quad M_r \Big( B^0\big( x, R \big) \cap F_\omega  \Big)  \ > \ C \, \bigg( \frac{R}{r} \bigg)^{s-1/n} \Bigg\} \\ \\
 	&=&  \bigcap_{n \in \mathbb{N}} \  \bigcap_{C \in \mathbb{N}} \  \bigcap_{\rho \in \mathbb{Q}^+} \ \bigcup_{x \in X} \  \bigcup_{R \in \mathbb{Q} \cap (0,\rho)} \ \bigcup_{r \in \mathbb{Q} \cap (0,R)} \ \Bigg\{ \omega \in \Omega  \ : \  M_r \big( \beta^0_{x,R}(F_\omega )  \big)  \ > \ C \, \bigg( \frac{R}{r} \bigg)^{s-1/n} \Bigg\} \\ \\
 	&=&  \bigcap_{n \in \mathbb{N}} \  \bigcap_{C \in \mathbb{N}} \  \bigcap_{\rho \in \mathbb{Q}^+} \ \bigcup_{x \in X} \  \bigcup_{R \in \mathbb{Q} \cap (0,\rho)} \ \bigcup_{r \in \mathbb{Q} \cap (0,R)} \ \Xi^{-1}\,  \Big( \big(  C \, (R/r)^{s-1/n} , \infty \big) \Big).
 \end{eqnarray*}
 The set $ \Xi^{-1} \,  \Big( \big(  C \, (R/r)^{s-1/n}, \infty \big) \Big)$ is open by the lower semicontinuity of $ \Xi^{-1}$ and therefore $A$ is a $\mathcal{G}_{\delta}$ subset of $\Omega$.
 \\ \\
 To complete the proof that $A$ is residual, it remains to show that $A$ is dense in $\Omega$.  For $n \in \mathbb{N}$ let
 \[
 A_n \ = \ \{ \omega \in \Omega \ : \  \dim_\text{{A}} F_\omega \geq s - 1/n\}.
 \]
 It follows that $A_n$ is $\mathcal{G}_{\delta}$ by the same argument as above, and since
 \[
 A \ = \ \bigcap_{n \in \mathbb{N}} A_n
 \]
 it follows from the Baire Category Theorem that it suffices to show that $A_n$ is dense in $\Omega$ for all $n$.  Let $n \in \mathbb{N}$, $\omega=(\omega_1, \omega_2, \dots) \in \Omega$,  and $\varepsilon>0$.  Let  $u=(u_1, u_2, \dots) \in \Omega$ be such that $\dim_\text{A} F_{u} > s-1/n$, choose $k \in \mathbb{N}$ such that $2^{-k} <\varepsilon$ and let $v =(\omega_1,\dots, \omega_k,u_1, u_2, \dots)$.  It follows that $d(v, \omega) < \varepsilon$ and, furthermore,
 \[
 F_v =\bigcup_{j_1\in \mathcal{I}_{\omega_1}, \dots, j_k \in\mathcal{I}_{\omega_k}} S_{\omega_1, j_1} \circ \dots \circ S_{\omega_k,j_k}(F_u).
 \]
Since, for all $j_1\in \mathcal{I}_{\omega_1}, \dots, j_k \in\mathcal{I}_{\omega_k}$ the map $S_{\omega_1, j_1} \circ \dots \circ S_{\omega_k,j_k}$ is a bi-Lipschitz contraction, it follows from basic properties of the Assouad dimension that $\dim_\text{A} F_{v} \geq \dim_\text{A} F_{u} > s-1/n$ and so $v \in A_n$, proving that $A_n$ is dense.

\subsection{Proof of Theorem \ref{percolationassouad}} \label{percolationassouadproof}

The upper bound is trivial, and we prove the lower bound here.
As we condition on non-extinction, we may assume there exists $x\in F$ and hence also a sequence of nested compact cubes $Q_{k}^{x}$ that each contain $x$, have sidelengths equal to $n^{-k}$ and are such that $x=\cap_{k\in\N}Q_k^x$.
We start by introducing some additional notation. At the $(k+1)$th stage in the construction of $F$ the cube $Q_k^x$ was split into $N=n^d$ compact cubes.
We will index these cubes by $\mathcal{I}=\{1,2,\hdots,N\}$ (ordered lexicographically by their midpoints) and keep track of the tree structure of subcubes by words that give their position in the iteration.
That is for words of length $m$ we write $Q_k^x(w)$, where $w\in\mathcal{I}^m$, to mean the uniquely determined cube at the $(k+m)$th stage of the construction lying inside $Q_k^x$ at position $w$ starting from $Q_k^x$.
We also write $Q_k^x=Q_k^x(\varnothing)$.
Let $p_{\neg e}>0$ be the probability that any cube which has survived up to some point in the construction does not go on to become extinct.  Due to the independence and homogeneity of the construction, this is the same for any surviving cube at any level.  Moreover, it is strictly positive due to our assumption on $p$.
The following lemma is similar in spirit to Lemma \ref{exceptionalLemma}.
\newpage

\begin{lma}\label{borelcantelliperc}
Let $x$ be as above.
Almost surely there exists an increasing sequence of natural numbers $(M_i)_{i=1}^\infty$ such that, for all $i \in \N$,  all cubes
\[
Q_{M_i}^x(w)\text{ where } w \in\{\varnothing\}\cup\bigcup_{a=1}^{i} \mathcal{I}^a
\]
survive and each of the last cubes $\{Q_{M_i}^x(w)\}_{w \in \mathcal{I}^i}$ in this iteration do not become extinct.
\end{lma}
\begin{proof}
Let $m, r\in\N$ be given. First we establish the probability of all cubes $Q_r^x(w)$ for $w\in\{\varnothing\}\cup\bigcup_{a=1}^m \mathcal{I}^a$ surviving and not becoming extinct. By the  homogeneity of the construction the probability of those cubes surviving is independent of $r$ and is the number of `(weighted) coin tosses' needed for all cubes to survive. As we are given that at least one path (the one for $x$) survives, the number of `tosses' is
\[
L_N^m \ = \ \sum_{a=1}^{m} ( N^{a}-1 ) \ = \  \frac{N^{m+1}-N}{N-1} - m,
\]
and so the probability of all of the cubes surviving is $p^{L_N^m}$. We also have to take into account the non-extinction criteria. Given that they have survived to the $(r+m)$th level, the probability that all of the cubes $\{Q_{r}^x(w)\}_{w \in \mathcal{I}^m}$ will not become extinct is $p_{\neg e}^{N^m-1}$.  Thus the probability of all cubes $Q_r^x(w)$ for $w\in\{\varnothing\}\cup\bigcup_{a=1}^m \mathcal{I}^a$ surviving and not becoming extinct is $\hat p_{m}=p^{L_N^m}p_{\neg e}^{N^m-1}$.  Now define $l(m+1)=l(m)+k(m)$, where $l(1)=1$ and
\[
k(m) \ = \ m\left\lceil\frac{-\log2}{\log(1-\hat p_{m})}\right\rceil.
\]
Let $\mathcal E_m$ be the event
\begin{multline*}
\mathcal E_m=\left\{\vphantom{\bigcup_{a=1}^m}\text{for at least one of $j\in\{0,m,2m,\hdots,k(m)-m\}$ we have }\right.\\\left.\text{that all $Q_{l(m)+j}^x(v)$ survive and are non-extinct in the limit for $v\in\{\varnothing\}\cup\bigcup_{a=1}^m \mathcal{I}^a$}\right\}.
\end{multline*}
Given that the cubes $Q_k^x$ all survive, it is evident that the behaviour of one $k(m)/m$ block is independent of the next and so
\[
\mathbb P(\mathcal E_m) \ = \ 1-(1-\hat p_{m})^{k(m)/m} \ \geq \ 1/2.
\]
Lemma \ref{borelcantelliperc} now follows immediately by the Borel Cantelli Lemma and the fact that $\mathcal E_m$ are easily seen to be independent.
\end{proof}

Using Lemma \ref{borelcantelliperc} we now show that, almost surely conditioned on non-extinction, $X=[0,1]^d$ is a weak tangent to $F$.  The required lower bound on the dimension of $F$ then follows from Proposition \ref{weaktang00}.  Let $T_i$ be the homothetic similarity that maps the cube $Q_{M_i}^x$ to $X$.
By Lemma \ref{borelcantelliperc} we have that, almost surely conditioned on non-extinction, each of the subcubes $Q_{M_{i}}^x(v)$ for $v\in\mathcal{I}^i$ survive and are non-empty in the limit.
Now $T_i (F)\cap X$ is the union of all blow ups of these subcubes under $T_i$ and, since each blown up subcube contains at least one point and has  diameter $\sqrt{d} \,  n^{-i}$, it follows that
\[
d_{\mathcal H}(T_i (F)\cap X, X) \, \leq  \, \sqrt{d} \,  n^{-i}
\]
and so $d_{\mathcal H}(T_i (F)\cap X, X)\to 0$ as $i\to\infty$ as required.
\\ \\
The optimal projection result now follows as a simple consequence of $F$ being almost surely of full dimension. In particular, for all $k \leq d$ and $\pi \in \Pi_{d,k}$ we have
	\[
	F \subset \pi F \times \pi^\perp,
	\]
	where $\pi^\perp$ is the $(d-k)$-dimensional orthogonal complement of (the $k$ dimensional subspace identified with) $\pi$, and so by basic properties of how Assouad dimension behaves concerning products \cite[Lemma 9.7]{Robinson11} it follows that, for all realisations where $\dim_\text{A} F = d$, we have
	\[
	d \ = \ \dim_\text{A} F \  \leq \  \dim_\text{A} \pi F \, + \, \dim_\text{A} \pi^\perp  \ =  \ \dim_\text{A} \pi F \, + \, d-k
	\]
	which gives $\dim_\text{A} \pi F \geq k$.  The opposite inequality is trivial and since $\dim_\text{A} F = d$ occurs almost surely conditioned on non-extinction, the result follows.

\section{Questions and discussion of results} \label{questionsection}

It will be clear to the diligent reader that our methods could be used with only minor alterations to prove more general results than the ones we chose to state.  This was a conscious choice made in order to clearly display what we believe are the key new phenomena we were able to observe.  In particular, in the random self-similar setting, a simple modification of the `good sets' would yield that for the almost sure lower bound, the maximum over deterministic IFSs can be replaced with the supremum over periodic words.  This only effects things when there are overlaps as Theorem \ref{theo2-1} shows that the two values are the same if the UOSC is assumed.
\begin{ques}
Assuming no separation conditions, is the almost sure Assouad dimension of a 1-variable random self-similar attractor given by the supremum of the Assouad dimensions of attractors corresponding to eventually periodic words?
\end{ques}
For the example in Section \ref{counterexample} we readily obtain that the almost sure Assouad dimension is 1.  Moreover, in the self-affine setting, our methods should enable the analogous results to be proved for random Lalley-Gatzouras carpets or random Bara\'nski carpets.  In these cases the upper bounds would be considerably more technical, but by combining our methods with the techniques in \cite{Fraser14a} the expected results should follow.
\\ \\
Our methods for proving Theorem~\ref{percolationassouad} easily extend to encompass more general percolation models, for example the model considered by Falconer and Jin~\cite[Section 6]{Falconer14}. This model is based on an IFS of similarities $\{S_{i}\}_{i\in \mathcal I}$ satisfying the OSC and a probability vector $\textbf{\emph{p}}$ associated with the power set of $\{S_{i}\}_{i\in \mathcal I}$ where we assume each entry is non-zero.  Each sub-IFS describes a different way to iterate the construction starting from any cylinder and for every surviving $k$th level cylinder, which IFS to apply is chosen randomly and independently with respect to this probability vector, giving rise to a random subset of the initial attractor.  Similar to Theorem~\ref{percolationassouad} we can show that if the probabilities are chosen such that there is a positive probability of non-extinction then, conditioned on this happening, the Assouad dimension is almost surely maximal, i.e it equals the Assouad dimension of the attractor of $\{S_{i}\}_{i\in \mathcal I}$.
\\ \\
This paper has unearthed the following general principle: \emph{for a randomly generated set, the Assouad dimension is generically as big as possible}.  We wonder to what other random settings this principle applies.  Some of the most famous examples of random fractals are certain sets associated to random functions or random processes. For example, the graph, range, and level sets of (fractional) Brownian motion are random fractals with interesting dimension theoretic properties.  The almost sure Hausdorff dimensions of these objects are long known and are, in some sense, intermediate values.
\begin{ques}
What are the almost sure Assouad dimensions of the graph, range and level sets for (fractional) Brownian motion on $[0,1]$?
\end{ques}
A natural direction to try to push Theorems \ref{SSresult} and \ref{lowerboundcarpets} would be to consider more general RIFSs, either by allowing more general classes of map or less control on the overlaps, or a more general random model, for example $V$-variable. One particular instance of this could be to consider more general self-affine carpets, like those considered in \cite{LalleyGatzouras92, Baranski07, Feng05, Fraser12}.  The Assouad dimension of Bedford-McMullen carpets and the extensions considered by Lalley-Gatzouras \cite{LalleyGatzouras92} and \cite{Baranski07} were computed in \cite{Mackay11, Fraser14a}, but extending the calculations to the carpets considered by Feng-Wang \cite{Feng05} or the first author \cite{Fraser12} is currently out of reach due to the lack of a suitable grid like structure.
\\ \\
Another possible way of randomising the IFS construction is to consider one particular deterministic IFS and then randomise some of the defining parameters.  There are many ways of doing this, but perhaps the most well known is the approach pioneered by Falconer in \cite{Falconer88}.  Here a fixed collection of linear contractions on $\mathbb{R}^n$ are turned into a nontrivial IFS by randomly choosing a set of corresponding translation vectors.  The main result of \cite{Falconer88} is a formula for the Hausdorff dimension which holds for almost all translations (assuming a mild technical condition on the norms of the matrices).  It was already asked by the first author what the generic Assouad dimension of a self-affine set in this setting is~\cite[Question 4.3]{Fraser14a}.  We raise it again because, if the general principle in this paper applies here, then the answer would be that the Assouad dimension is almost surely maximal, i.e. the same as the ambient spatial dimension, however, this cannot be the case.  For example, work in $\mathbb{R}$ and take an IFS with two (or more) similarity maps with very small contraction ratios.  Then there will be a positive measure set of translations where the resulting self-similar set satisfies the OSC and thus has Hausdorff and Assouad dimension equal and strictly less than 1.
\\ \\
It is of particular interest that Theorems \ref{SSresult} and \ref{lowerboundcarpets} do not depend on which specific Bernoulli measure we chose.  In particular, as long as the defining probability vector is strictly positive, then the Assouad dimension is almost surely maximal no matter what.  This is perhaps surprising since for all other notions of dimension the almost sure value depends crucially on the exact measure used.  This leads to the following natural question.
\begin{ques}
For which (Borel probability) measures $\mu$ on $\Omega$ do Theorems \ref{SSresult} and \ref{lowerboundcarpets} remain true?
\end{ques}
We feel it is likely that this will hold for much more general measures than Bernoulli measures.  The argument should go through with only minor modifications for Gibbs measures for H\"older potentials, for example, but we expect a much more general statement is true.
\\ \\
Concerning fractal percolation, we wonder about various types of intersection problems. For example, it is currently a topic of interest to study slices, that is intersections with translations of lower dimensional subspaces, see \cite{Rams14}.
\begin{ques}
Can one say anything about the almost sure (conditioned on non-extinction) Assouad dimension of (almost all of) the slices of fractal percolation?
\end{ques}

\section*{Acknowledgements}

This work began whilst the three authors were attending the \emph{Conference in honour of Kenneth Falconer's 60th birthday} in May 2014 hosted by INRIA (Paris).
We thank both the Universities of St Andrews and Warwick for their hospitality when hosting research visits related to this project. We also thank Xiong Jin for making helpful comments on the paper.\\

JMF was financially supported by the EPSRC grant EP/J013560/1 whilst employed at the University of Warwick. JJM was partially supported by the NNSF of China (no. 11201152), the Fund for the Doctoral Program of Higher Education of China (no. 20120076120001) and SRF for ROCS, SEM (no. 01207427).  ST was financially supported by the EPSRC Doctoral Training Grant AMC3-DTG012.


\begin{thebibliography}{99}

\bibitem[A]{Aikawa91}
H. Aikawa.
Quasiadditivity of Riesz capacity,
{\em Math. Scand.}, {\bf 69}, (1991), 15--30. 

\bibitem[As1]{assouadphd}
P. Assouad.
Espaces m{\'e}triques, plongements, facteurs,
{\em Th\`ese de doctorat d'\'Etat, Publ. Math. Orsay 223--7769, Univ. Paris XI, Orsay}, (1977).

\bibitem[As2]{Assouad79}
P. Assouad.
\'Etude d'une dimension m\'etrique li\'ee \`a la possibilit\'e de plongements dans $\mathbb{R}^n$,
{\em C. R. Acad. Sci. Paris S\'{e}r. A-B}, {\bf 288}, (1979), 731--734.

\bibitem[BG]{Bandt92}
C. Bandt and S. Graf.
Self-similar sets. {VII}. {A} characterization of self-similar fractals with positive {H}ausdorff measure,
 {\em Proc. Amer. Math. Soc.}, {\bf 114}, (1992), 995--1001.

\bibitem[B]{Baranski07}
K.~Bara\'nski.
Hausdorff dimension of the limit sets of some planar geometric constructions,
{\em Adv. Math.}, {\bf 210}, (2007), 215--245.

\bibitem[BHS1]{Barnsley05}
M. F. Barnsley, J. Hutchinson and {\"O}. Stenflo.
 A fractal valued random iteration algorithm and fractal hierarchy,
 {\em Fractals}, {\bf 13}, (2005), 111--146.

\bibitem[BHS2]{Barnsley08}
M. F. Barnsley, J. E. Hutchinson and {\"O}. Stenflo.
V-variable fractals: fractals with partial self similarity,
{\em Adv. Math.}, {\bf 218}, (2008), 2051--2088.


\bibitem[BHS3]{Barnsley12}
M.~Barnsley, J.~E. Hutchinson, and {\"O}.~Stenflo.
\newblock {$V$}-variable fractals: dimension results.
\newblock {\em Forum Math.}, {\bf 24}, (2012), 445--470.


\bibitem[Ba]{superfractals}
M. F. Barnsley.
{\em Superfractals},
Cambridge University Press, Cambridge, 2006.

\bibitem[Be]{Bedford84}
T. Bedford.
 Crinkly curves, Markov partitions and box dimensions in self-similar sets,
 {\em Ph.D dissertation, University of Warwick}, (1984).



\bibitem[F1]{Falconer88}
K.~J. Falconer.
 The Hausdorff dimension of self-affine fractals,
 {\em Math. Proc. Camb. Phil. Soc.}, {\bf 103}, (1988), 339--350.


\bibitem[F2]{FractalGeo}
K.~J. Falconer.
{\em Fractal Geometry: Mathematical Foundations and Applications},
John Wiley, 2nd Ed., 2003.


\bibitem[F3]{TecFracGeo}
K.~J. Falconer.
 {\em Techniques in Fractal Geometry},
John Wiley, 1997.

\bibitem[FJ]{Falconer14}
K. J. Falconer and X. Jin.
Exact dimensionality and projections of random self-similar measures and sets,
\emph{J. London Math. Soc.}, {\bf 90}, (2014), 388--412.

\bibitem[FW]{Feng05}
D.-J. Feng and Y. Wang.
A class of self-affine sets and self-affine measures,
{\em J. Fourier Anal. Appl.}, {\bf11}, (2005), 107--124.

\bibitem[Fr1]{Fraser12}
J. M. Fraser.
On the packing dimension of box-like self-affine sets in the plane,
\emph{Nonlinearity}, {\bf 25}, (2012), 2075--2092.


\bibitem[Fr2]{Fraser12b}
J. M. Fraser.
Dimension and measure for typical random fractals,
\emph{Ergodic Th. Dyn. Syst.}, {\bf 35}, (2015), 854--882.


\bibitem[Fr3]{Fraser14a}
J.~M. Fraser.
 Assouad type dimensions and homogeneity of fractals,
 {\em Trans. Amer. Math. Soc.}, {\bf 366}, (2014), 6687--6733.

\bibitem[FHOR]{Fraser14}
J.~M. {Fraser}, A.~M. {Henderson}, E.~J. {Olson} and J.~C. {Robinson}.
 On the Assouad dimension of self-similar sets with overlaps,
 {\em Adv. Math.}, {\bf 273}, (2015), 188--214.

\bibitem[FO]{Fraser11}
J.~M. Fraser and L. Olsen.
 Multifractal spectra of random self-affine multifractal Sierpi\'nski sponges in $\mathbb{R}^d$,
 {\em Indiana Univ. Math. J.}, {\bf 60}, (2011), 937--984.



\bibitem[FS]{FraserShmerkin}
J.~M. {Fraser} and P. Shmerkin.
 On the dimensions of a family of overlapping self-affine carpets,
{\em Ergodic Th. Dyn. Syst., to appear}, available at: arXiv:1405.4919.



\bibitem[Fu]{Furstenberg08}
H. Furstenberg.
Ergodic fractal measures and dimension conservation,
\emph{Ergodic Th. Dyn. Syst.}, {\bf 28}, (2008), 405--422.





\bibitem[GL]{Gui08}
Y. Gui and W. Li.
 A random version of {M}c{M}ullen-{B}edford general {S}ierpinski carpets and its application,
 {\em Nonlinearity}, {\bf 21}, (2008), 1745--1758.

\bibitem[H]{Heinonen01}
J. Heinonen.
\emph{Lectures on Analysis on Metric Spaces}, Springer-Verlag, New York, 2001.

\bibitem[Hu]{Hutchinson81}
J.~E. Hutchinson.
 Fractals and self-similarity,
{\em Indiana Univ. Math. J.}, {\bf 30}, (1981), 713--747.



\bibitem[HLOPS]{Hyde10}
J. T. Hyde, V. Laschos, L. Olsen, I. Petrykiewicz and A. Shaw.
 Iterated Ces\`aro averages, frequencies of digits, and Baire category,
\emph{Acta Arith.}, {\bf 144}, (2010), 287--293.




\bibitem[K]{Kaenmaki13}
A. K\"aenm\"aki, J. Lehrb\"ack and M. Vuorinen.
Dimensions, Whitney covers, and tubular neighborhoods,
\emph{Indiana Univ. Math. J.}, {\bf 62}, (2013), 1861--1889.



\bibitem[KZ]{Koskela03}
P. Koskela and X. Zhong.
Hardy's inequality and the boundary size,
\emph{Proc. Amer. Math. Soc.}, {\bf 131}, (2003), 1151--1158. 

\bibitem[LG1]{LalleyGatzouras92}
S.~P. {Lalley} and D. {Gatzouras}.
Hausdorff and box dimensions of certain self-affine fractals,
 {\em Indiana Univ. Math. J.}, {\bf 41}, (1992), 533--568.

\bibitem[LG2]{LalleyGatzouras94}
S.~P. {Lalley} and D. {Gatzouras}.
 Statistically self-affine sets: Hausdorff and box dimensions,
 {\em J. Theoret. Probab.}, {\bf 7}, (1994), 437--468.

\bibitem[LT]{Lehrback13}
J. Lehrb\"ack and H. Tuominen.
A note on the dimensions of Assouad and Aikawa,
\emph{J. Math. Soc. Japan}, {\bf 65}, (2013), 343--356. 

\bibitem[LW]{Liu02}
Y.-Y. Liu and J. Wu.
A dimensional result for random self-similar sets,
\emph{Proc. Amer. Math. Soc.}, {\bf 130}, (2002), 2125--2131.




\bibitem[LN]{Lau99}
K.-S. Lau and S.-M. Ngai.
 Multifractal measures and a weak separation condition,
 {\em Adv. Math.}, {\bf141}, (1999), 45--96.

\bibitem[LLMX]{Li14}
W.-W. Li, W.-X. Li, J. J. Miao and L.-F. Xi,
Assouad dimensions of Moran sets and Cantor-like sets,
\emph{preprint} (2014), available at: arXiv:1404.4409v3.

\bibitem[L]{Luukkainen98}
J. Luukkainen.
Assouad dimension: antifractal metrization, porous sets, and homogeneous measures,
\emph{J. Korean Math. Soc.}, {\bf 35}, (1998), 23--76.

\bibitem[M]{Mackay11}
J. M. Mackay.
Assouad dimension of self-affine carpets,
\emph{Conform. Geom. Dyn.} {\bf 15}, (2011), 177--187.

\bibitem[MT]{mackaytyson}
J. M. Mackay and J. T. Tyson.
\emph{Conformal dimension. Theory and application},
University Lecture Series, 54. American Mathematical Society, Providence, RI, 2010.



\bibitem[Ma]{Mandelbrot74}
B.B. Mandelbrot.
Intermittent turbulence in self-similar cascades-divergence of high moments and dimension of the carrier,
{\em Journal of Fluid Mechanics}, {\bf 62}, (1974), 331--358.


\bibitem[Mc]{McMullen84}
C. McMullen.
 The Hausdorff dimension of general Sierpi\'nski carpets,
 {\em Nagoya Math. J.}, {\bf 96}, (1984), 1--9.

\bibitem[Mo]{Moran46}
P.~A.~P. Moran.
 Additive functions of intervals and {H}ausdorff measure,
 {\em Proc. Cambridge Philos. Soc.}, {\bf42}, (1946), 15--23.

\bibitem[O]{Olsen11a}
L. Olsen.
On the Assouad dimension of graph directed Moran fractals,
\emph{Fractals}, {\bf 19}, (2011), 221--226.

\bibitem[Ol]{Olson02}
E. Olson.
Bouligand dimension and almost Lipschitz embeddings,
\emph{Pacific J. Math.}, {\bf 202}, (2002), 459--474.

\bibitem[ORS]{Olson14}
E. Olson, J. C. Robinson and N. Sharples.
Generalised Cantor sets and the dimension of products,
 {\em preprint} (2014), available at: arXiv:1407.0676.



\bibitem[OR]{Olson10}
E. J. Olson and J. C. Robinson.  Almost Bi-Lipschitz Embeddings and Almost Homogeneous Sets,
{\it Trans. Amer. Math. Soc.}, {\bf 362}:1, (2010), 145--168.


\bibitem[Ox]{oxtoby}
J. C. Oxtoby. \emph{Measure and Category},  Springer, 2nd Ed., 1996.


\bibitem[RS]{Rams14}
M. Rams and K. Simon.
The geometry of fractal percolation,
in {\em Geometry and Analysis of Fractals} D.-J. Feng and K.-S. Lau (eds.), pp 303--324,
{\it  Springer Proceedings in Mathematics \& Statistics.} {\bf 88}, Springer-Verlag, Berlin Heidelberg, 2014. 


\bibitem[R]{Robinson11}
J. C. Robinson.
{\em Dimensions, Embeddings, and Attractors},
Cambridge University Press, 2011.



\bibitem[S]{Salat66}
T. \v Sal\'at.
A remark on normal numbers,
\emph{Rev. Roumaine Math. Pures Appl.}, {\bf 11}, (1966), 53--56.


\bibitem[Sc]{Schief94}
A. Schief.
Separation properties for self-similar sets,
\emph{Proc. Amer. Math. Soc.}, {\bf 122}, (1994), 111--115.

\bibitem[Z]{Zerner96}
M. P.~W. Zerner.
 Weak separation properties for self-similar sets,
 {\em Proc. Amer. Math. Soc.}, {\bf124}, (1996), 3529--3539.

\end{thebibliography}

\end{document}